\documentclass[12pt]{amsart}
\usepackage[utf8]{inputenc}
\usepackage[total={6.5in,9in},top=1in, left=1in, right=1in, bottom=1in]{geometry}
\usepackage{fancyhdr}
\usepackage{libertine}
\usepackage[libertine]{newtxmath}
\usepackage{color,amsmath,amssymb,mathtools,graphicx}
\usepackage{bbold}
\usepackage{epsfig,fancyhdr}
\usepackage{dsfont} 
\usepackage{graphicx,amsmath,mathtools,amssymb,float}
\usepackage{epsfig,color,fancyhdr,setspace}
\usepackage{dsfont}
\usepackage{epstopdf}
\usepackage{import}
\usepackage{lineno}
\usepackage{stackrel,dsfont}
\usepackage[allcolors = blue,colorlinks]{hyperref}
\usepackage[abs]{overpic}
\usepackage[center]{caption}
\usepackage{subcaption}

\newtheorem{theorem}{Theorem}[section]
\newtheorem{lemma}[theorem]{Lemma}
\newtheorem{definition}[theorem]{Definition}
\newtheorem{example}[theorem]{Example}
\newtheorem{proposition}[theorem]{Proposition}
\newtheorem{remark}[theorem]{Remark}
\newtheorem{corollary}[theorem]{Corollary}
\newtheorem{conjecture}[theorem]{Conjecture}

\title{A Generalization of the Gram determinant of type A}

\author{Rhea Palak Bakshi}
\address{Department of Mathematics, The George Washington University, Washington DC, USA}
\email{rhea\_palak@gwu.edu}
\author{Dionne Ibarra}
\address{Department of Mathematics, The George Washington University, Washington DC, USA}
\email{dfkunkel@gwu.edu}
\author{Sujoy Mukherjee}
\address{Department of Mathematics, The George Washington University, Washington DC, USA}
\email{sujoymukherjee@gwu.edu}
\author{J\'{o}zef H. Przytycki}
\address{Department of Mathematics, The George Washington University, Washington DC, USA and \linebreak Department of Mathematics, University of Gda\'{n}sk, Poland}
\email{przytyck@gwu.edu}
\date{\today}

\begin{document}

\begin{abstract}
    The Gram determinant of type $A$ was introduced by Lickorish in his work on invariants of 3 - manifolds. We generalize the theory of the Gram determinant of type $A$ by evaluating, in the annulus, a bilinear form of non-intersecting connections in the disc. The main result provides a closed formula for this Gram determinant. We conclude the paper by discussing Chen's conjecture about the Gram determinant of type $Mb$ evaluated in the Klein bottle. 
\end{abstract}

\maketitle
\tableofcontents
\section{Introduction}

Modern work on Gram determinants in knot theory began with the construction of the Witten-Reshetikhin-Turaev invariants of 3-manifolds by Lickorish. \cite{Wit,RT,Li4}. Lickorish's Gram determinant (also called the Gram determinant of type $A$) was computed in the series of papers \cite{KS}, \cite{Wes}, and \cite{DiF}. For our work, Cai's proof of the closed formula of the determinant using Jones - Wenzl idempotents is important \cite{Cai}. 

In fact, in knot theory, matrices whose determinants are related to the Gram determinants, are prevalent. The earliest bilinear forms and their related Gram determinants (or their modifications) in knot theory are as follows:
\begin{enumerate}
    \item the Seifert matrix of a link, $L$, denoted by $S_L$, and its modification which gives the Alexander matrix $tS_L-t^{-1}S_L^{T}$. Recall that the determinant of the Alexander matrix is the Alexander polynomial of the link $L,$ 
    \item  the determinant of the Goeritz matrix of a link which is called the determinant of the link. If the determinant is not zero, then it is the rank of the first homology group of the double branched cover of $\mathbb{S}^3$ branched along the link, and 
    \item the linking matrix of a  framed link. If the determinant is not zero, then it is the rank of the first homology group of the 3 - manifold obtained from surgery along the link. 
\end{enumerate} 

Rodica Simion introduced the Gram determinant of type $B$ and its closed form was computed in \cite{Ch-P} and \cite{M-Sa}.

\subsection{Relative Skein Modules}

Relative skein modules allow the theory of the Gram determinants to be developed for $3$ - manifolds using any skein relation (for example, the HOMFLYPT skein relation). See \cite{Prz1,Prz2}. In our work, we use the relative Kauffman bracket skein module to study Gram determinants. 

\begin{definition}\

Let $M$ be an oriented $3$ - manifold and  $\{x_1,x_2,\ldots,x_{2n}\}$ be a set of $2n$ framed points on $\partial M$. Let $\mathcal{L}_{fr}(2n)$ be the set of all relative framed links in $(M, \partial M)$ considered up to ambient isotopy keeping $\partial M$ fixed, such that $L \cap \partial M = \partial L = \{x_i\}$. Let $R$ be a commutative ring with unity,  $A$ an invertible element in $R$, and $S_{2,\infty}(2n)$, the submodule of $R\mathcal{L}_{fr}(2n)$, generated by all the Kauffman bracket skein relations. Then, the relative Kauffman bracket skein module of $M$ is the quotient: $$\mathcal{S}_{2,\infty}(M, \{x_i\}_1^{2n}; R, A) = \frac{R\mathcal{L}_{fr}(2n)}{ S_{2,\infty}(2n)}.$$ 

\end{definition}

\begin{theorem}\cite{Prz2}

Let $F$ be a surface and $M = F \times I$ (or  $F\ \hat{\times}\ I$, if $F$ is unoriented), such that $\partial F \neq \varnothing$. Then, 
$\mathcal{S}_{2,\infty}(M, \{x_i\}_1^{2n}; R, A)$ is a free $R$-module whose basis consists of relative links in $F$, without trivial components. 

\end{theorem}

The following corollary discusses well known combinatorial notions using relative skein modules.

\begin{corollary}\label{impcor}\
\begin{enumerate}
    \item $\mathcal{S}_{2,\infty}(D^2 \times I, \{x_i\}_1^{2n}; R, A)$ is a free $R$-module with $c_n = \frac{1}{n+1}\binom{2n}{n}$ basic elements. Here $c_n$ denotes the $n^{th}$ Catalan number. 
    \item $\mathcal{S}_{2,\infty}(A^2 \times I, \{x_i\}_1^{2n}; R, A)$ is a free $R[z]$ - module with ${2n\choose n}$ basic elements, where $z$ denotes the homotopically non - trivial curve on the annulus.
    
\end{enumerate}

\end{corollary}

\subsection{The Temperley - Lieb algebra and the Jones - Wenzl idempotent}

The relative Kauffman bracket skein module of $D^2 \times I$ can be equipped with an algebra structure which leads to the classical Temperley - Lieb algebra (see \cite{TL}).

\begin{definition}\

Let $R$ be a commutative ring with unity and an invertible element $A$. Further, let $d \ = -A^{-2} - A^2 \in R.$ The Temperley - Lieb Algebra, denoted by $TL_n(d)$, is an $R$ - algebra generated by $n$ elements, $\{ \mathbb{1}, e_1, e_2,\ldots, e_{n-1}\}$ with relations:

\begin{enumerate}
    \item $e_i^2 = de_i$ for $1 \leq i \leq n-1$,
    \item $e_ie_je_i = e_i$ $ \forall |i - j| = 1$, and
    \item $e_ie_j = e_je_i$ $ \forall |i - j| > 1$. 
\end{enumerate}
 
\end{definition}

Diagrammatically, the elements of the basis of $TL_n(d)$, as an $R$ - module, can be represented as crossingless connections between the $2n$ marked points on the boundary of a rectangle with $n$ points each on its top edge and bottom edge (see Figure \ref{Temp}). For further reading, see \cite{Kau}.

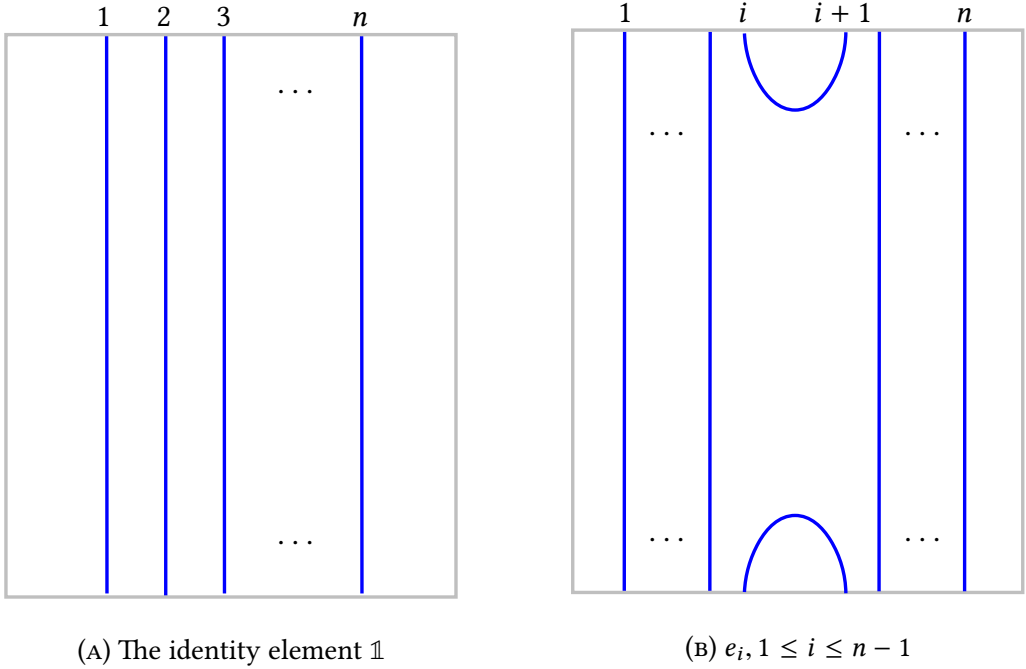
\begin{figure}\label{Temp}
\centering
\begin{subfigure}{.45\textwidth}
\centering
\begin{overpic}[unit=1mm, scale = 1.5]{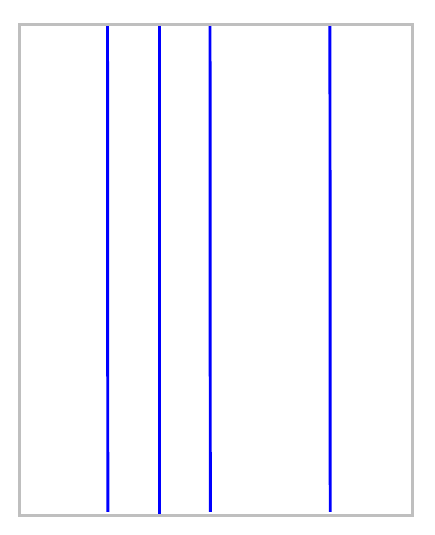}

\put(15,79){$1$}
\put(23,79){$2$}
\put(31,79){$3$}
\put(39,70){$\hdots$}
\put(39,10){$\hdots$}
\put(49,79){$n$}

\end{overpic}
\caption{The identity element $\mathbb{1}$} \label{fig1}
\end{subfigure}
\begin{subfigure}{.45\textwidth}
\centering
\begin{overpic}[unit=1mm, scale = 1.5]{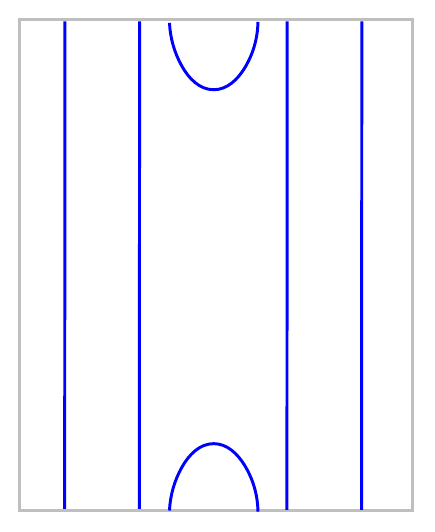}

\put(9,79){$1$}
\put(25,79){$i$}
\put(35,79){$i+1$}
\put(54,79){$n$}
\put(47,64){$\hdots$}
\put(47,10){$\hdots$}
\put(13,64){$\hdots$}
\put(13,10){$\hdots$}

\end{overpic}
\caption{$e_i, 1 \leq i \leq  n-1$} \label{fig2}
\end{subfigure}
\caption{The generators of $TL_n(d).$}\label{Temp}
\end{figure}





There is a natural bilinear form on the elements of $TL_n(d)$ which Lickorish introduced to construct the Witten - Reshetikhin - Turaev invariants of $3$ - manifolds. The determinant of this bilinear form is called the Gram Determinant of type $A$ which we will discuss in the following subsection.

\begin{definition}\cite{Jon1, Wen}\

Let $p: \mathbb{B}_n \longrightarrow S_n$ be a map, where $\mathbb{B}_n$ denotes the Artin braid group and $S_n$, the permutation group. This function sends a braid word to the induced permutation. Define $F_n = \sum\limits_{\pi \in S_n} (A^3)^{|\pi|}b_{\pi} \in \mathbb{Z}[A^{\pm1}]\mathbb{B}_n$, where $|\pi|$ is the length of the permutation $\pi$ and $b_{\pi}$ is the unique, minimal, positive braid such that $p(b_{\pi}) = \pi$. The $n^{th}$ Jones - Wenzl idempotent, denoted by $f_n,$ is defined to be $\frac{F_n}{(\{n\}_{A^4})!} \in \mathbb{Q}(A)\mathbb{B}_n$, where $\{n\}_q = 1 + q + q^2 + \cdots + q^{n - 1}$, and $(\{n\}_q)! = \{1\}_q\cdot\{2\}_q\cdots\{n\}_q$.\footnote{The extension of $\mathbb{Z}[A^{\pm1}]$ to the field of rational functions, $\mathbb{Q}(A),$ does not change the Gram determinant.}

\end{definition}

The Jones - Wenzl idempotent is evaluated in $TL_n(d)$ by taking the quotient of $\mathbb{Q}(A)\mathbb{B}_n$ by the Kauffman bracket skein relations.

\begin{theorem}\cite{Jon1, KL, Lic}\

  The $n^{th}$ Jones - Wenzl idempotent, $f_n \in TL_n(d),$ satisfies the following properties:

\begin{enumerate}
    \item $f_ne_i = 0 =  e_if_n $, for $1 \leq i \leq n-1$,
    \item $(f_n - \mathbb{1})$ belongs to the algebra generated by $\{e_1, e_2, \ldots, e_{n-1}\}$, and \item $f_nf_n = f_n$. 
\end{enumerate}

\end{theorem} 
The $n^{th}$ Jones-Wenzl idempotent is represented by a rectangular box (see Figure \ref{fn}) with $n$ curves entering and $n$ curves exiting it. The integer $n$ beside a curve represents $n$ parallel copies of it.

\begin{figure}[ht]
\centering
\begin{overpic}[unit=1mm, scale = .45]{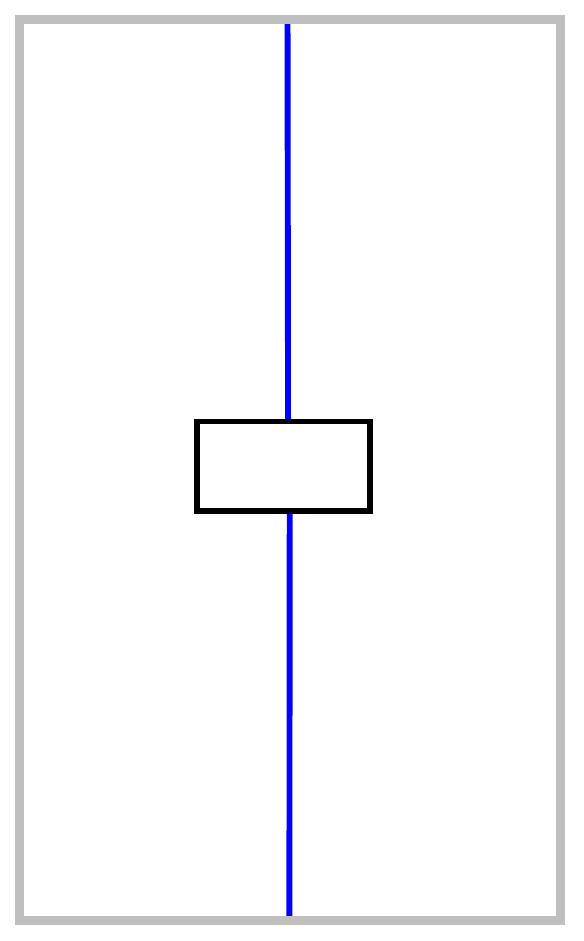}
\put(10,9){$n$}
\end{overpic}
\caption{The Jones - Wenzl idempotent $f_n.$}\label{fn}
\end{figure}

\begin{theorem} \cite{Wen}\ 

The Jones - Wenzl idempotent $f_{n+1},$ for $n > 1,$ satisfies recursive relation shown in Figure \ref{wenzl}.
\begin{figure}[ht] 
\centering
$$ \vcenter{\hbox{
\centering
\begin{overpic}[unit=1mm, scale = .45]{fn}

\put(3.5,9){$n+1$}
\end{overpic}
}} =  \vcenter{\hbox{
\centering
\begin{overpic}[unit=1mm, scale = .45]{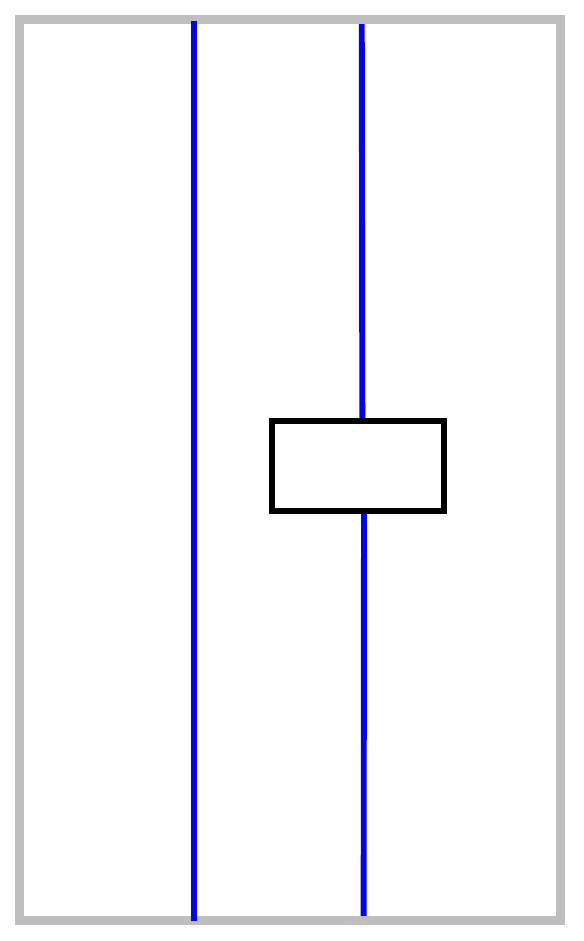}

\put(13,9){$n$}
\end{overpic}
}}  - \frac{\Delta_{n-1}}{\Delta_n}  \vcenter{\hbox{
\centering
\begin{overpic}[unit=1mm, scale = .45]{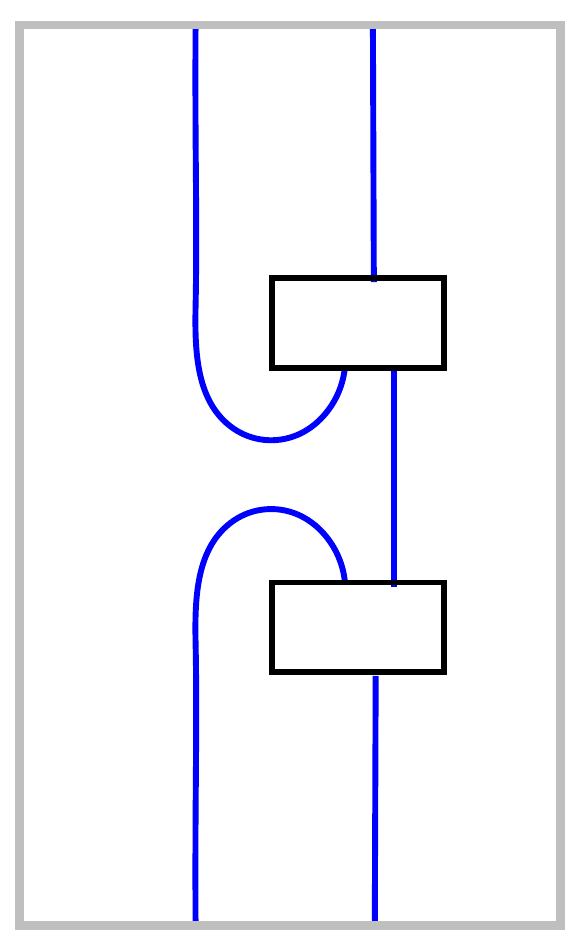}

\put(18.5,6){$n$}
\put(18.5,21){\footnotesize{$n-1$}}
\put(18.5,35){$n$}
\end{overpic}
}} $$
\caption{The Wenzl formula for $f_{n+1}.$}
\label{wenzl}
\end{figure}
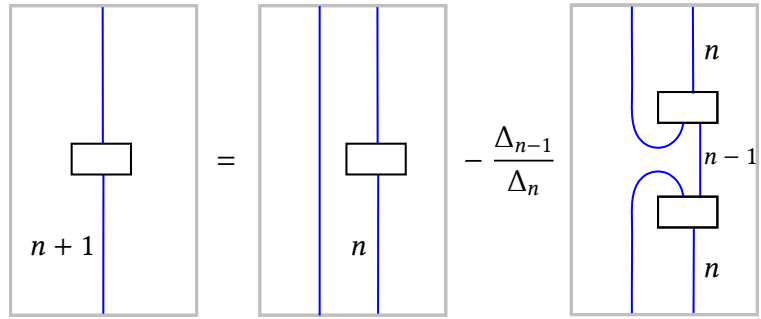

\end{theorem}

\subsection{The Gram Determinant of Type A}

\begin{definition}\

Consider the disc $D^2$ with $2n$ framed points on its boundary. Let $\mathcal{A}_n = \{a_1, a_2,\ldots, a_{c_n}\}$ be the set of all diagrams with crossingless connections, up to ambient isotopy, between the $2n$ framed points in $D^2$. 
Define a bilinear form $\langle \ , \ \rangle$ in the following way: $$\langle \ , \ \rangle : \mathcal{S}_{2,\infty}(D^2 \times I, \{x_i\}_1^{2n}) \times \mathcal{S}_{2,\infty}(D^2 \times I, \{x_i\}_1^{2n}) \longrightarrow R.$$

Let $a_i, a_j \in \mathcal{A}_n.$ Glue $a_i$ with the inversion of $a_j$ along the marked circle, respecting the labels of the framed points. The resulting picture is that of a disc with disjoint null homotopic circles. Thus, we define, $\langle a_i , a_j\rangle = d^m$ where m denotes the number of these circles.  Example \ref{exam18} illustrates an example of the bilinear form when $n = 4$.

The Gram matrix of type $A$ is defined as $G_n^{A} = ( \langle a_i , a_j\rangle ) _{1 \leq i,j \leq c_n} $. Its determinant $D_n^{A}$ is called the Gram determinant of type $A$.

\end{definition}

\begin{example}\label{exam18}
$$ \left\langle \vcenter{\hbox{\includegraphics[scale = .25]{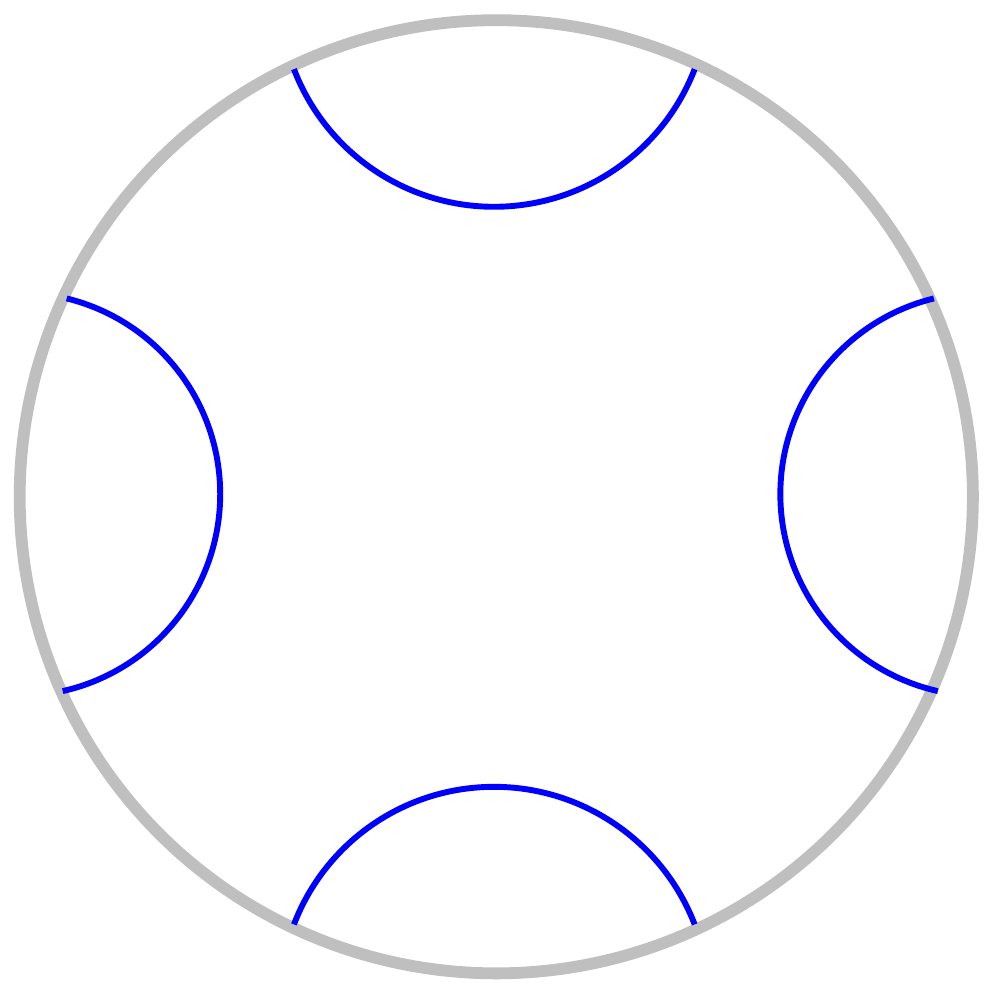}}} , \vcenter{\hbox{\includegraphics[scale = .25]{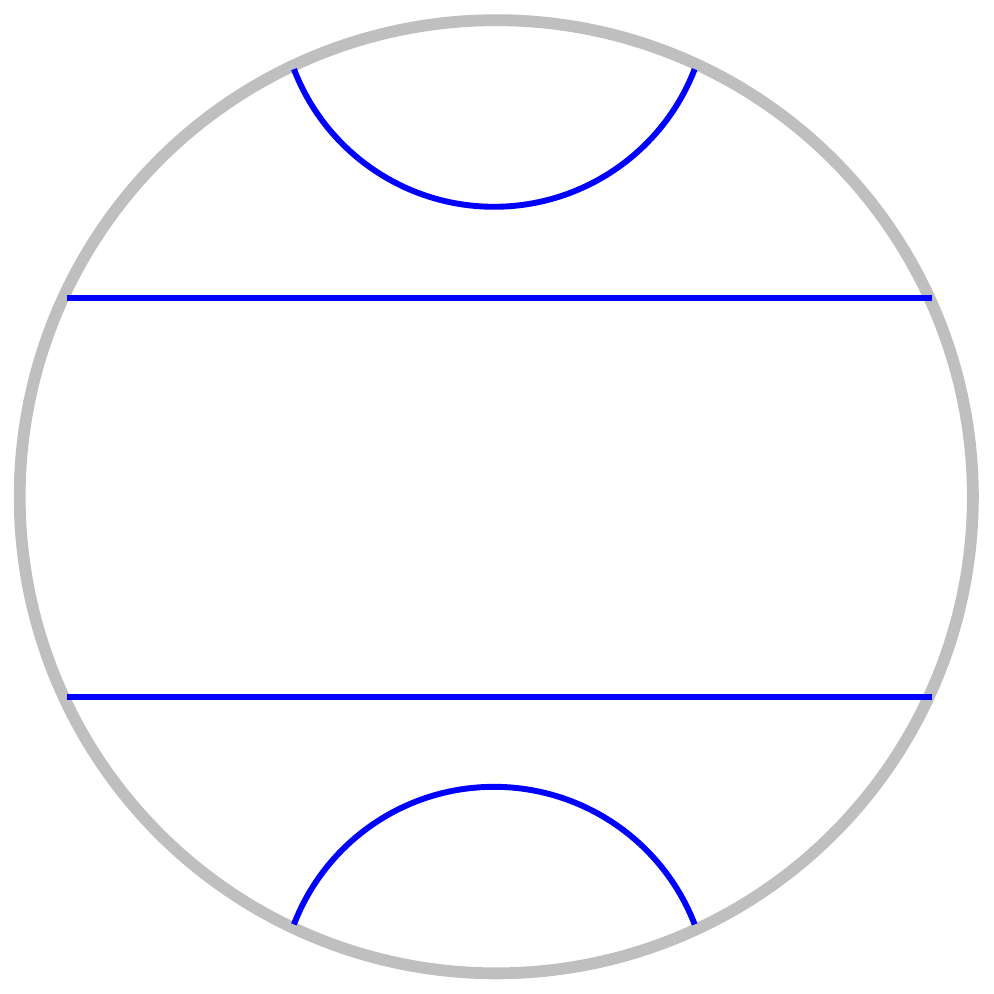}}} \right\rangle = \vcenter{\hbox{\includegraphics[scale = .25]{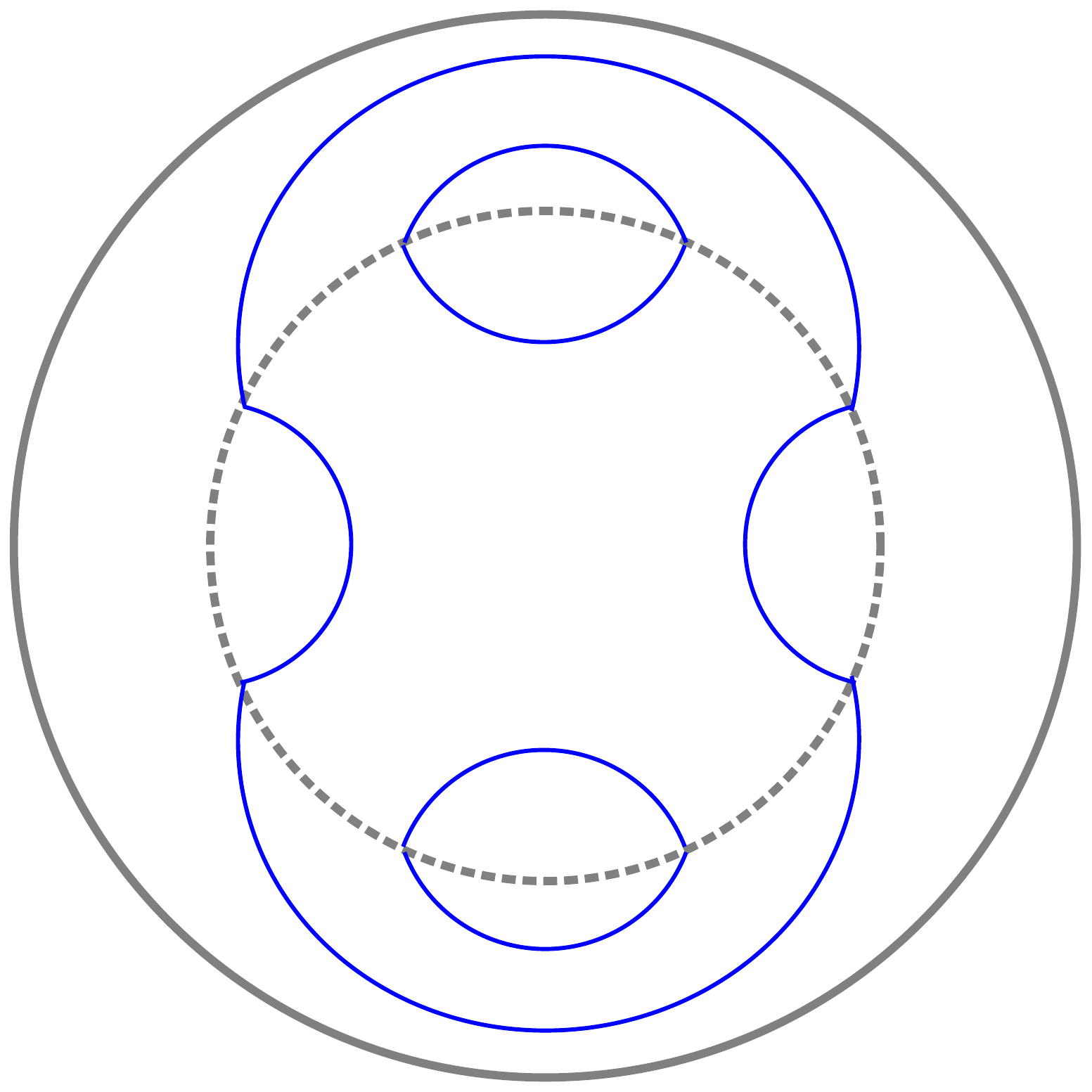}}} = d^3$$
\end{example}

\begin{theorem}\cite{Wes, DiF, Cai}\

Let $R = \mathbb{Z}[A^{\pm 1}].$ Then,
\begin{equation*}
D_n^{A}(d) = \Delta _1^{c_n}\prod\limits_{i=1}^{n}(\frac{\Delta _i}{\Delta_{i-1}})^{\alpha _i}, \text{where }  \Delta_i = (-1)^i \frac{A^{2i+2}-A^{-2i-2}}{A^2-A^{-2}} \text{ and } \alpha _i = {{2n}\choose {n-i}} - {{2n} \choose {n - i - 1}}.
\end{equation*}
\end{theorem}

Note that $\Delta _1 = -A^2 - A^{-2} = d.$ Furthermore, $\Delta_i(d)$ is the Chebyshev polynomial of the second kind (see Theorem \ref{GenA}). In \cite{Cai}, the author used Jones - Wenzl idempotents to construct a new basis of $TL_n(d)$ and provided a new proof of the formula for the Gram determinant of type $A$ using it.

\subsection{The Gram Determinant of Type B}
The Gram determinant of type $B$ was first considered by Rodica Simion while working with matrices of chromatic joins \cite{Sim1, Sch}. Jones - Wenzl idempotents and Chebyshev polynomials of the first kind were used to find the closed formula of the Gram determinant of type $B$  in \cite{Ch-P}.

\begin{definition}\

Let $A^2$ be an annulus with $2n$ marked points on its outer boundary. Let $\mathcal{B}_n = \{b_1, b_2, \ldots, b_{2n \choose n}\}$ be the set of all diagrams of crossingless connections between these $2n$ points. Define a bilinear form $\langle \ , \ \rangle$ in the following way: $$\langle \ , \ \rangle : \mathcal{S}_{2,\infty}(A^2 \times I, \{x_i\}_1^{2n}; R, A) \times \mathcal{S}_{2,\infty}(A^2 \times I, \{x_i\}_1^{2n}; R, A) \longrightarrow R[z].$$

Given $b_i, b_j \in \mathcal{B}_n$, glue $b_i$ with the inversion of $b_j$ along the marked circle, respecting the labels of the marked points. The resulting picture has disjoint circles which are either homotopically non - trivial or null homotopic. Then, $\langle b_i , b_j\rangle = z^k d^m$, where $k$ and $m$ denote the number of these circles, respectively.

The Gram matrix of type $B$ is defined as $G_n^{B} = (\langle b_i , b_j\rangle) _{1 \leq i, j \leq {2n \choose n}}.$ Its determinant $D_n^B$ is called the Gram determinant of type $B$.

\end{definition}

The following theorem, evaluating this determinant, answers Rodica Simion's question. 

\begin{theorem}\cite{Ch-P,M-Sa}\

Let $R = \mathbb{Z}[A^{\pm 1}].$ Then, $D_n^{B}(d,z)=  \prod \limits_{i=1}^n(T_i(d)^2-z^2)^{2n \choose n-i},$ where $T_i$ denotes the Chebyshev polynomial of the first kind defined recursively by the equation $T_{n+1}(d) = d\cdot T_n(d) - T_{n-1}(d)$, with the initial conditions $T_0(d) = 2$ and $T_1(d) = d$. 
        
\end{theorem}

The idea of the proof is as follows. For a positive integer $n$, there are ${2n \choose n}$ annular Catalan states. 
The evaluation of the Gram matrix of type $B$ at $z = (-1)^{k-1}T_k(d)$ is closely related to the evaluation of the Gram matrix of type $A.$ A Hopf link is formed with one of its components, the annulus used in the Gram determinant of type $B$, and the other decorated with the $k^{th}$ Jones - Wenzl idempotent. For details, see \cite{Ch-P}. With this decoration, the number of annular Catalan states is equal to ${2n \choose n} - {2n \choose n-k}$ 
and thus, the nullity of the matrix is at least ${2n \choose n-k}$. Therefore, $D_n^{B}(d, z)$ is divisible by $(T_k(d)+(-1)^kz)^{2n\choose n-k}$. By a simple analysis      
of involution, which changes $z$ to $-z$, it follows that $D_n^{B}(d,z)$ is also divisible by $(T_k(d)-(-1)^kz)^{2n\choose n-k}$. Thus, the theorem      
holds up to a multiplicative factor. By analysing the term having the highest degree which is obtained from the diagonal of the matrix, it follows that this factor is equal to $1$. 

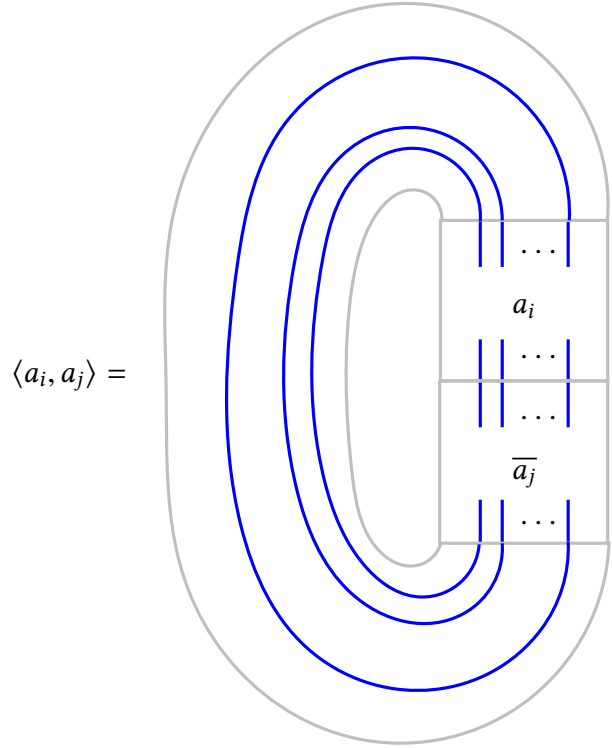
\begin{figure}[ht]
\centering
$$ \langle a_i, a_j \rangle = 
\vcenter{\hbox{
\begin{overpic}[unit=1mm, scale = 1.4]{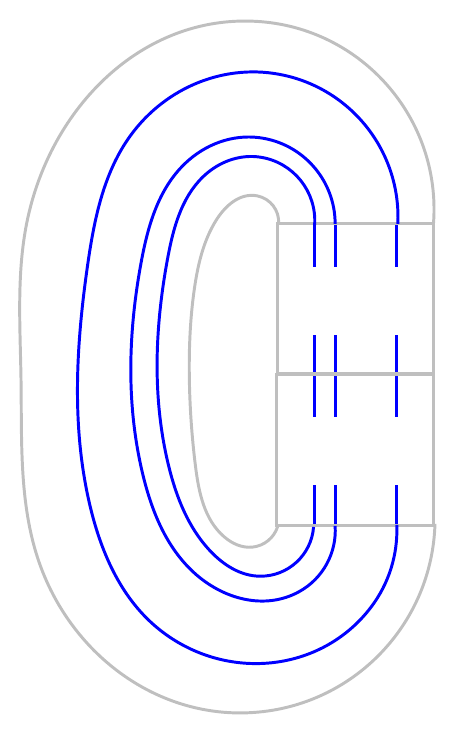}

\put(50,68){$\hdots$}
\put(50,54){$\hdots$}
\put(49,60){$a_i$}
\put(50,46){$\hdots$}
\put(49, 38){$\overline{a_j}$}
\put(50,32){$\hdots$}
\end{overpic} }}
$$
\caption{Depiction of the bilinear form $\langle a_i , a_j \rangle = Tr(a_i\overline{a_j})$ of generalized type $A$, where $a_i \in \mathcal{A}_n$ and $\overline{a_j}$ is the reflection of $a_j \in \mathcal{A}_n$ about the horizontal axis.}\label{GenerlizedBilinear}
\end{figure}

\section{The Gram Determinant of Generalized Type A}

We now introduce a generalized version of the Gram determinant of type $A$. 

\begin{definition}\

Let the disc $D^2$, with $2n$ marked points on its boundary, be considered as a rectangle with $n$ points on the top edge and $n$ points on the bottom edge. Define a bilinear form $\langle \ , \ \rangle$ in the following way: 
$$\langle \ , \ \rangle: \mathcal{S}_{2,\infty}(D^2 \times I, \{x_i\}_1^{2n}) \times \mathcal{S}_{2,\infty}(D^2 \times I, \{x_i\}_1^{2n}) \longrightarrow \mathbb{Z}[d, z].$$

 For $a_i, a_j \in \mathcal{A}_n$, glue $a_i$ with the reflection, about the horizontal axis, of $a_j$ which is denoted by $\bar a_j$, such that the bottom edge of $a_i$ is identified with the top edge of $\bar a_j$. Connect the marked points on the top edge of $a_i$ with those on the bottom edge of $\bar a_j$, in the annulus, respecting the ordering of the marked points (see Figure \ref{GenerlizedBilinear}). Recall that the described operation is the same as taking the trace of the image of the bilinear form.

The result is an annulus with two types of disjoint circles, homotopically trivial and non - trivial. Thus, we define, $\langle a_i , a_j\rangle = d^kz^m$ where $k$ and $m$ denote the number of these circles respectively.
We define the Gram matrix of generalized type $A$ as $G_n^{A^{gen}} = ( \langle a_i , a_j\rangle ) _{1 \leq i,j \leq c_n}$, and denote its determinant by  $D_n^{A^{gen}}$.

\end{definition}

\begin{remark}\

The Temperley - Lieb algebra leads to the Frobenius algebra by defining the Frobenius form to be the trace of the elements of the Temperley - Lieb algebra (see \cite{Koc}).

\end{remark}

\begin{theorem}\ \label{GenA}\

Let $R = \mathbb{Z}[A^{\pm 1}].$ Then,
the Gram determinant of generalized type $A$ is given by the following formula:
$$D^{A^{gen}}_n(d,z) = D^A_n(d) \prod\limits_{i=0} ^{\lfloor{\frac{n}{2}}\rfloor}\big(\frac{S_{n-2i}(z)}{\Delta_{n-2i}(d)}\big)^{\big( {n \choose i}-{n \choose i-1}\big)^2},$$ where $S_k(z)$ denotes the Chebyshev polynomial of the second kind, defined recursively by the equation $S_{k + 1}(z) = z\cdot S_k(z) - S_{k - 1}(z)$, with the initial conditions $S_0(z) = 1$ and $S_1(z) = z$.

\end{theorem}

Before giving the proof of this theorem in Section \ref{Main Proof}, we give an outline here.  To compute the Gram determinant $D_n^{A^{gen}}$, we change the basis of the Temperley - Lieb alegbra, $TL_n(d)$, so that in the new basis, the Gram matrix is a diagonal matrix. It follows from \cite{Cai} that the change of basis is given by an upper triangular matrix with $1$'s on the diagonal. Therefore, the Gram determinant is unchanged by the change of basis. To find the diagonal entries of the new Gram matrix, we use the standard properties of theta nets. The difference with the Gram determinant of generalized type $A$ is that is that we finally evaluate the theta nets in the annulus and not in the disc. Hence, we use the fact that the trace of the Jones - Wenzl idempotents, in the annulus, is the Chebyshev polynomial of the second kind in the variable $z$. Finally, we use the combinatorial properties of mountain paths to find a closed formula for the Gram determinant of generalized type $A$.

\subsection{Theta Nets}

This subsection serves as a preparation for the proof of the main theorem. 

\begin{definition}\ \label{def4}\ 

Let $a = y +z$, $b = x + z$, and $c = x+ y$, that is, $x = \frac{(b + c - a)}{2}$, $y = \frac{(a + c - b)}{2}$, and $z = \frac{(a + b - c)}{2}$, so that
$$\vcenter{\hbox{
\centering
\begin{overpic}[unit=1mm, scale = .5]{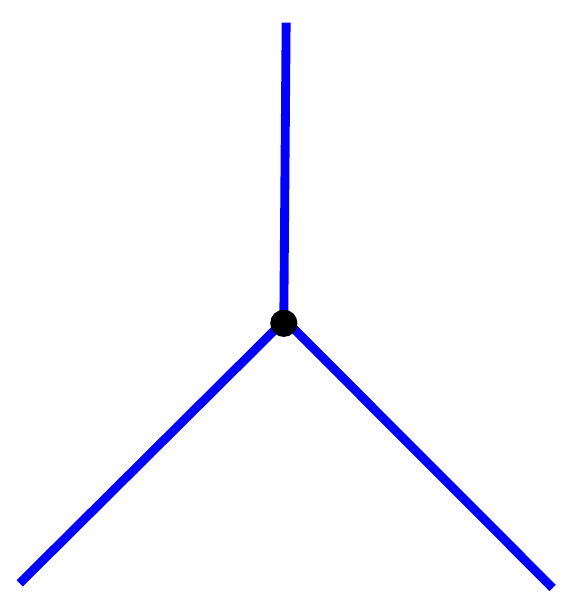}
\put(16,25){$c$}
\put(22,0){$b$}
\put(-3,0){$a$}
\end{overpic}
}} =  \vcenter{\hbox{
\centering
\begin{overpic}[unit=1mm, scale = .5]{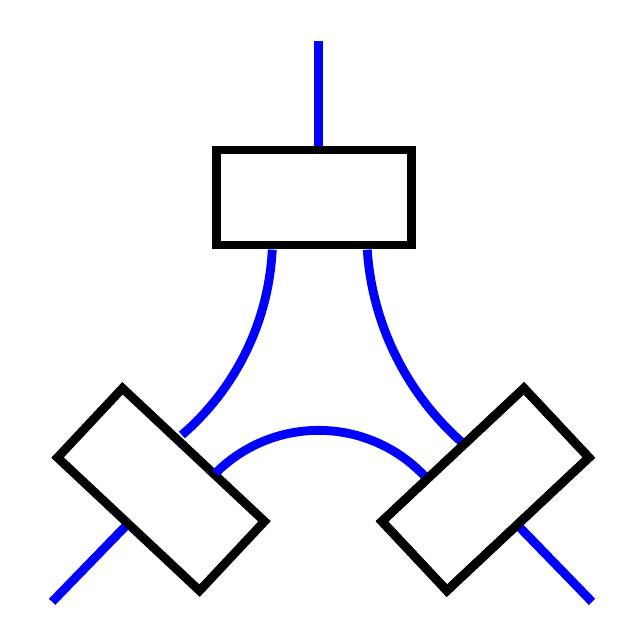}

\put(17,26){$c$}
\put(25,0){$b$}
\put(-1,0){$a$}
\put(21.5,13){$x$}
\put(9,13){$y$}
\put(15,6){$z$}
\end{overpic}
}}.
$$ 
Notice that $(a+b-c)$, $(b+c-a)$, and $(a+c-b)$ are all even integers. The Temperley - Lieb category $\pmb{TL}$ consists of objects $n \in \mathbb{N},$ morphism spaces $TL_{m,n},$ and multiplication $TL_{m,n} \times TL_{l,m} \longrightarrow TL_{l,n}$ (see, for example, \cite{QW}). Then, the following figure depicts the Markov trace of the multiplication of          $TL_{c,a+b} \times TL_{a+b,c} \longrightarrow TL_{a+b,a+b}.$\footnote{Recall that trace in the disk is called the Markov trace.}

\end{definition} 

$$  Tr \left\langle \  
\vcenter{\hbox{
\centering
\begin{overpic}[unit=1mm, scale = .3]{3vertex}
\put(11,15){$c$}
\put(12,0){$b$}
\put(-2,0){$a$}
\end{overpic} }}, \ \vcenter{\hbox{
\centering
\begin{overpic}[unit=1mm, scale = .3]{3vertex}
\put(11,15){$c$}
\put(12,0){$b$}
\put(-2,0){$a$}
\end{overpic} }} 
\right\rangle = \vcenter{\hbox{
\centering
\begin{overpic}[unit=1mm, scale = .7,angle = 90]{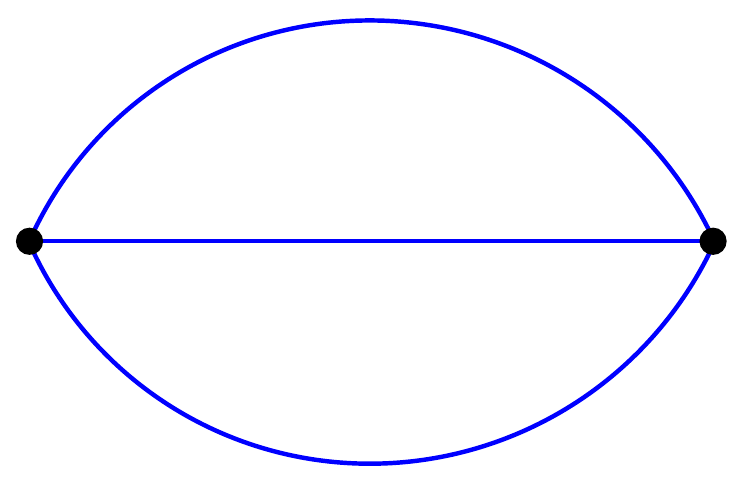}

\put(3,24.5){$c$}
\put(18.5,24.5){$a$}
\put(29.5,24.25){$b$}
\end{overpic}
}}.$$

Therefore, we obtain the following picture and can compute its unreduced Kauffman bracket using Theorem \ref{Iknow}.

\begin{equation*}
   \Theta(a,b,c) = \Gamma_{\mathbb{R}^2} (x,y,z) =
   \vcenter{\hbox{
\ \\ \ \\
\centering
\begin{overpic}[unit=1mm, scale = .31]{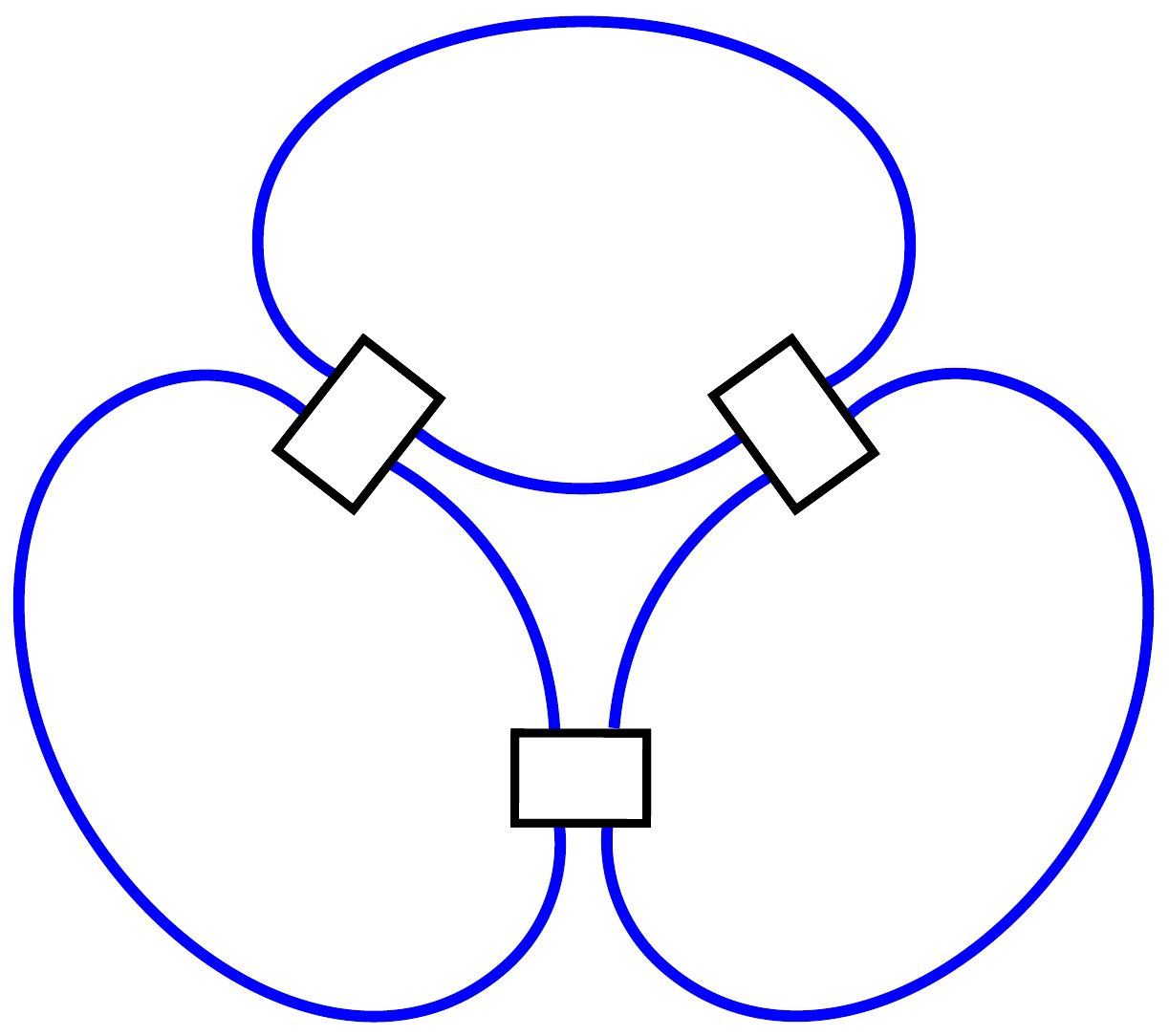}

\put(3,1){$y$}
\put(32.75,1){$z$}
\put(19,34.5){$x$}
\put(9, 15){$c$}
\put(23, 22.5){$b$}
\put(22.5, 7.5){$a$}
\end{overpic}
}}.
\end{equation*}

\begin{theorem} [\cite{MV,KL,Lic}]\label{Iknow}\ 

Let $\Delta_n !$ denote the product $\Delta_n \Delta_{n-1} \cdots \Delta_1.$
Then

$$ \Gamma_{\mathbb{R}^2} (x,y,z) = \frac{ \Delta_{x+y+z}! \Delta_{x-1}! \Delta_{y-1}! \Delta_{z-1}!}{\Delta_{x+y-1}! \Delta_{x+z-1}! \Delta_{y+z-1}!}. $$
\end{theorem}

In particular, for $b = x+z = 1$, that is, $|a-c|=1$, we get $\Gamma_{\mathbb{R}^2}(x,y,z) = \Delta_{y+1}$.

\begin{corollary} \label{theta} \label{theta network}\

The following formula reduces digons in the net (called bubble reduction),
$$\vcenter{\hbox{
\ \\ \ \\
\centering
\begin{overpic}[unit=1mm, scale = .7]{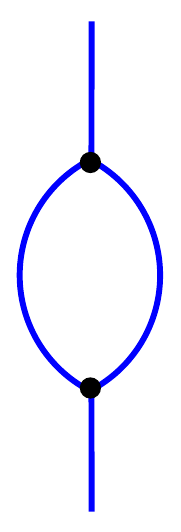}

\put(9,4){$a$}
\put(2.5, 17){$b$}
\put(12, 17){$c$}
\put(9,32){$a'$}
\end{overpic}
}} = \frac{\delta_{a,a'} \Theta(a,b,c)}{\Delta_a} \vcenter{\hbox{
\ \\ \ \\
\centering
\begin{overpic}[unit=1mm, scale = .7]{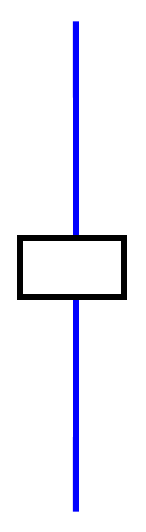}

\put(9,4){$a$}
\end{overpic}
}},$$

where $\delta_{a,a'}$ is the Kronecker delta. Compare with \cite{KL, Lic}. 

\end{corollary}

\begin{proof}\

If $a \neq a'$ then without loss of generality we can assume that $a' > a$. Therefore, the corresponding diagram has a return on the part with $a'$ in it. So the value of $a$ associated to the diagram on the left is equal to $0$. Otherwise $a = a'$, and by the property of Jones - Wenzl idempotents, we get \\

$$\vcenter{\hbox{
\centering
\begin{overpic}[unit=1mm, scale = .7]{opennet}

\put(9,4){$a$}
\put(2.5, 17){$b$}
\put(12, 17){$c$}
\put(9,32){$a'$}
\end{overpic}
}} = \lambda  \vcenter{\hbox{
\ \\ \ \\
\centering
\begin{overpic}[unit=1mm, scale = .7]{fnopen}

\put(9,4){$a$}
\end{overpic}
}}.$$
After closing the diagram in $\mathbb{R}^2$, the left side is a decorated theta curve with evaluation $\Theta (a,b,c)$, while the right side gives us $\lambda \Delta_a$. Thus, $\lambda = \frac{\Theta (a,b,c)}{\Delta_a}.$

\end{proof}

\begin{corollary}\label{Annulus}\

The formula for the decorated theta curve in $\mathbb{R}^2$ can be naturally generalized to the decorated theta curve in the annulus. That is,  $$\Gamma_{Ann}(x,y,z)=\frac{S_{x+y}(x)}{\Delta_{x+y}}\Gamma_{\mathbb{R}^2}(x,y,z),$$ 
	where we start from a pair of pants with holes $\partial_x,\partial_y$ and $\partial_z$. Then, we cap off $\partial_z,$ so that, $\partial_x$ is homotopic to $\partial_y$.

\end{corollary}

\begin{proof}\

Using Corollary \ref{theta network}, we get
$$\vcenter{\hbox{
\ \\
\centering
\begin{overpic}[unit=1mm, scale = .4]{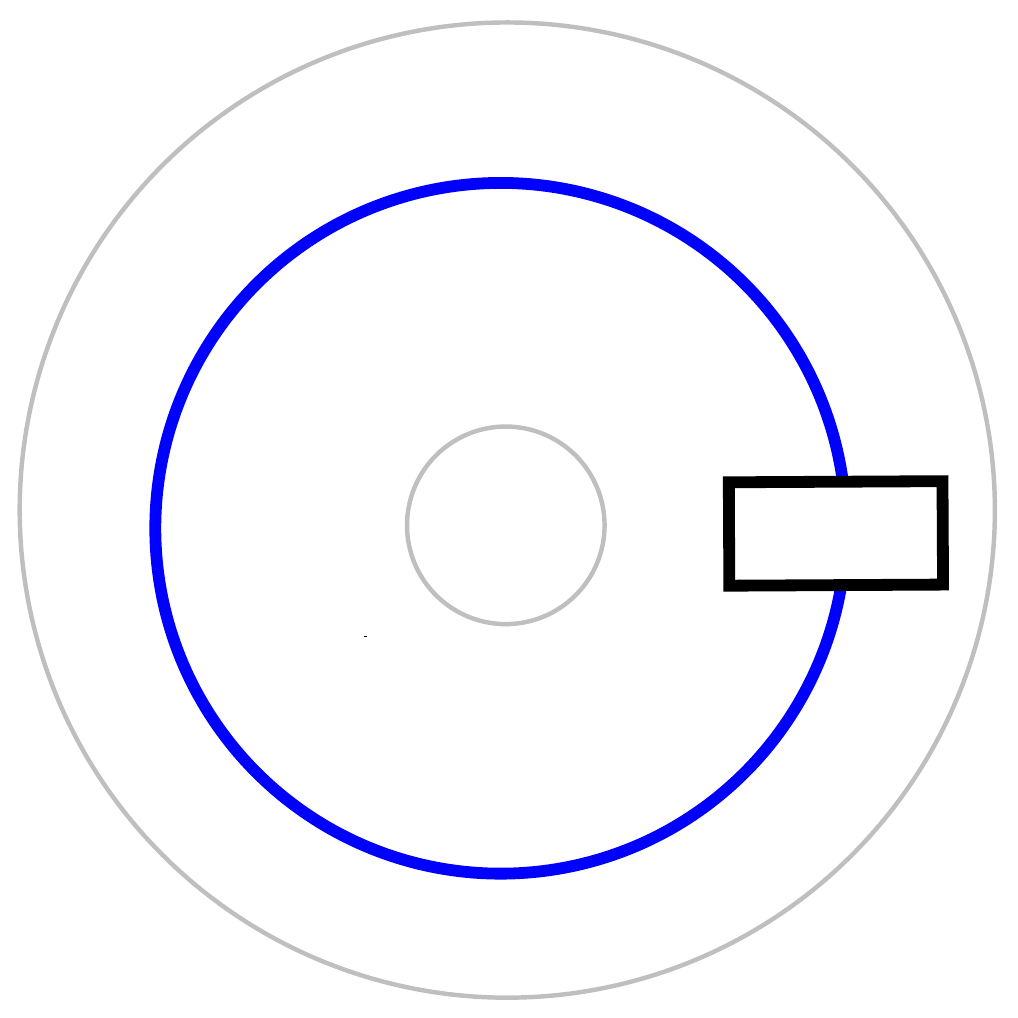}

\put(16,36){$x+y$}
\end{overpic}
}}$$

which in the Kauffman bracket skein module, $S_{2, \infty}(A^2 \times I)$, is equal to the Chebyshev polynomial $ S_{x + y}(\alpha)\frac{\Theta (a,b,c)}{\Delta_{x+y}}$, where $a = y+z$, $b = x+z$, and $c = x+y$.  
\end{proof}

\subsection{A Change of Basis for the Temperley - Lieb Algebra}

Following \cite{Cai}, we make a change of basis so that our bilinear form is orthogonal. Hence, in the new basis, the Gram matrix is diagonal. 

\begin{definition}\label{basis}\

Consider the finite sequence $\{a_1,a_2,\ldots,a_{2n-1}\}$ of natural numbers which satisfies the following two conditions:

\begin{enumerate} 

\item $a_1 = a_{2n-1} = 1$ and 
\item $| a_i - a_{i-1}| = 1$ $\forall i$.

\end{enumerate}
The elements of $TL_n(d)$, denoted by  $D_{a_1,a_2,\ldots,a_{2n-1}}$, are depicted in Figure \ref{Admissibletriplepoint}. The triple points in the diagram are admissible. 

\end{definition}

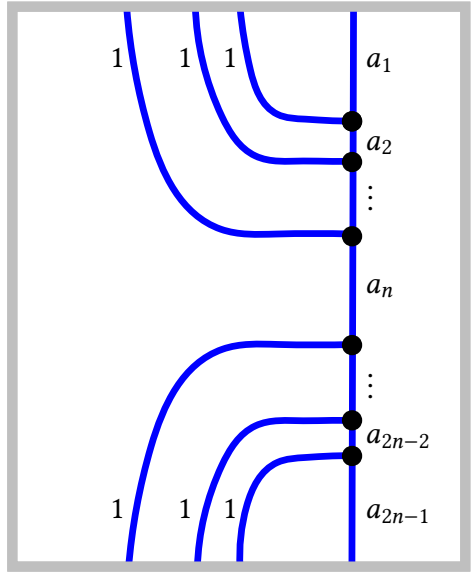
\begin{figure}[ht]
\ \\ \ \\
\centering
\begin{overpic}[unit=1mm, scale = 1.5]{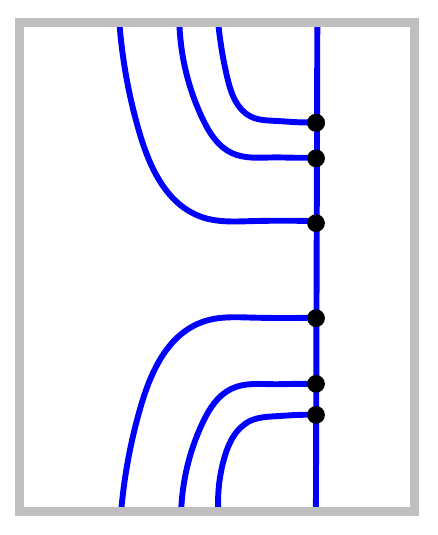}

\put(31, 70){1}
\put(25, 70){1}
\put(16, 70){1}
\put(50,70){$a_1$}
\put(50,59){$a_2$}
\put(50,51){$\vdots$}
\put(50,40){$a_n$}
\put(50,26){$\vdots$}
\put(50,20){$a_{2n-2}$}
\put(50,10){$a_{2n-1}$}
\put(31, 10){1}
\put(25, 10){1}
\put(16, 10){1}
\end{overpic}
\caption{The depiction of $D_{a_1,a_2,\ldots,a_{2n-1}} \in TL_n(d)$ using $a_1, \ldots, a_{2n-1}$ (Definition \ref{basis}).}\label{Admissibletriplepoint}
\end{figure}

\begin{example}\

Pictured below is the basis for $TL_3(d)$, decorated with Jones - Wenzl idempotents. 

$$ \vcenter{\hbox{
\ \\ \ \\
\centering
\begin{overpic}[unit=1mm, scale = .7]{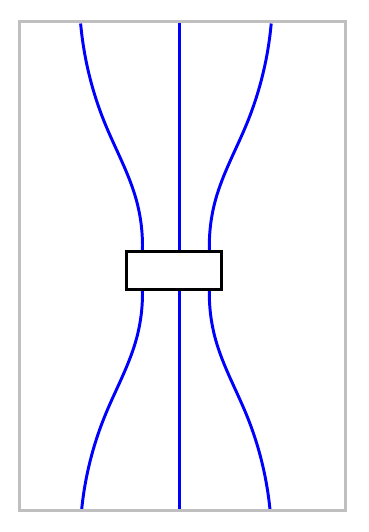}

\put(6,-2){$D_{1,2,3,2,1}$}
\end{overpic}
}} \ 
\vcenter{\hbox{
\ \\ \ \\
\centering
\begin{overpic}[unit=1mm, scale = .7]{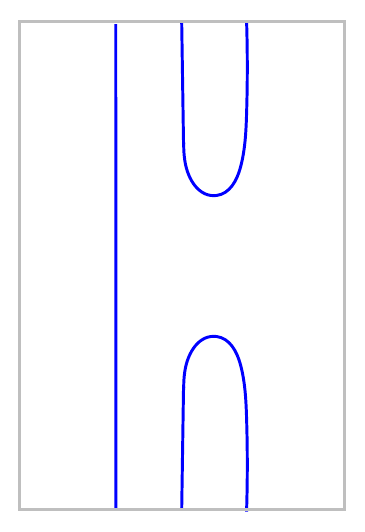}

\put(6,-2){$D_{1,0,1,0,1}$}
\end{overpic}
}} \ 
\vcenter{\hbox{
\ \\ \ \\
\centering
\begin{overpic}[unit=1mm, scale = .7]{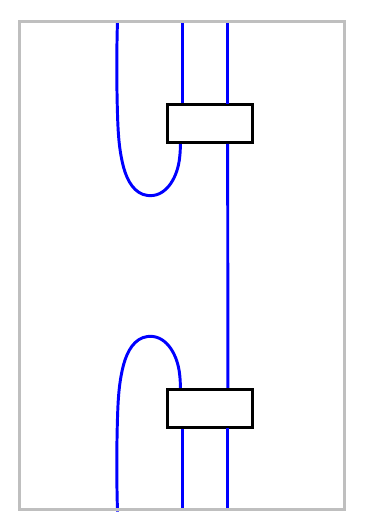}

\put(6,-2){$D_{1,2,1,2,1}$}
\end{overpic}
}} \ 
\vcenter{\hbox{
\ \\ \ \\
\centering
\begin{overpic}[unit=1mm, scale = .7]{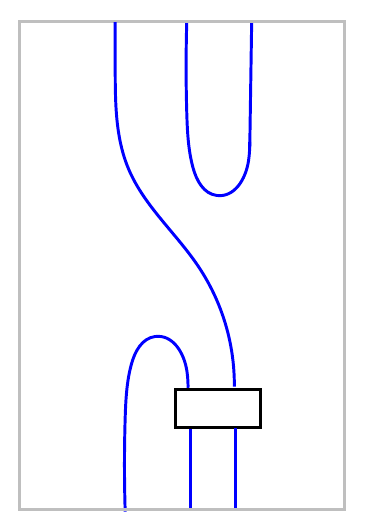}

\put(6,-2){$D_{1,0,1,2,1}$}
\end{overpic}
}} \ 
\vcenter{\hbox{
\ \\ \ \\
\centering
\begin{overpic}[unit=1mm, scale = .7]{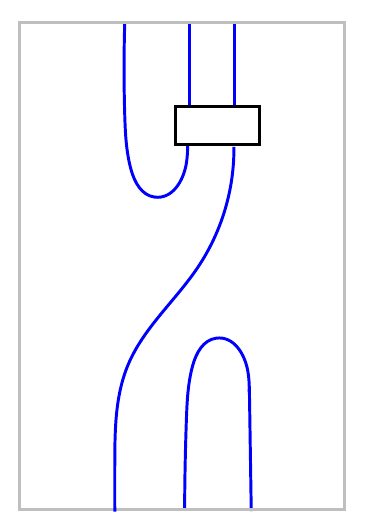}

\put(6,-2){$D_{1,2,1,0,1}$}
\end{overpic}
}}$$
\\

\end{example}

In order to compute the Gram determinant of generalized type $A$, we will need the following results.

\begin{theorem}\cite{Cai}\label{CBM}\

Let $\mathcal{D}_n$ be the collection of all elements $D_{a_1,a_2,\ldots,a_{2n-1}}$. Then, $\mathcal{D}_n$ is a basis for $TL_n(d)$. The change of basis matrix between $\mathcal{A}_n$ and $\mathcal{D}_n$ is an upper triangular matrix with $1$'s on the diagonal. In particular, the matrix has determinant 1.
\end{theorem}

\begin{theorem}\label{Orthogonal}\

The elements $D_{a_1,a_2,\ldots,a_{2n-1}}$ are orthogonal with respect to the bilinear form of generalized type $A$.

\end{theorem}

 \begin{proof}\
 
  We adapt the proof for the Gram determinant of type $A$ \cite{Cai} to the Gram determinant of generalized type $A$. In particular, we use Corollary \ref{theta network} for $a \neq a'$ to obtain orthogonality.
 \end{proof}

\begin{corollary}\label{DetUnchanged}\

The Gram determinant of generalized type $A$ is unchanged under the chance of basis from $\mathcal{A}_n$ to $\mathcal{D}_n$.

\end{corollary}

\begin{proof}\

The proof follows from Theorem \ref{CBM}. 
\end{proof}

\subsection{Proof of the Main Theorem}\label{Main Proof}
By Proposition \ref{Orthogonal} and Corollary \ref{DetUnchanged}, it suffices to find the diagonal entries of the Gram matrix of generalized type $A$ in the basis $\mathcal{D}_n$. We follow the proof in \cite{Cai} in the evaluation using Corollaries \ref{theta} and \ref{Annulus}.

Since we work in the annulus, we need to modify the part of Cai's proof which takes into account the terms appearing after bubble reduction (using Corollary \ref{theta}). 
The distinction in our case is that after bubble reduction, we obtain $S_{a_n}(z)$, while in Cai's proof the bubble reduction gives $S_{a_n}(d)$.

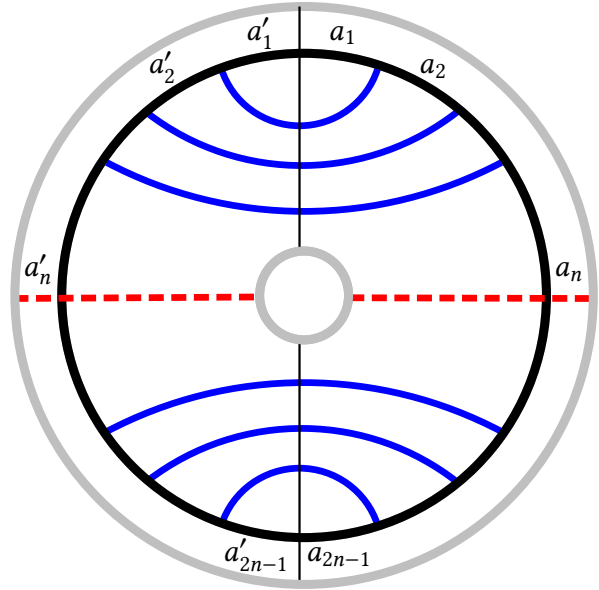
\begin{figure}[ht]
\centering
$$
\vcenter{\hbox{
\begin{overpic}[unit=1mm, scale = 1]{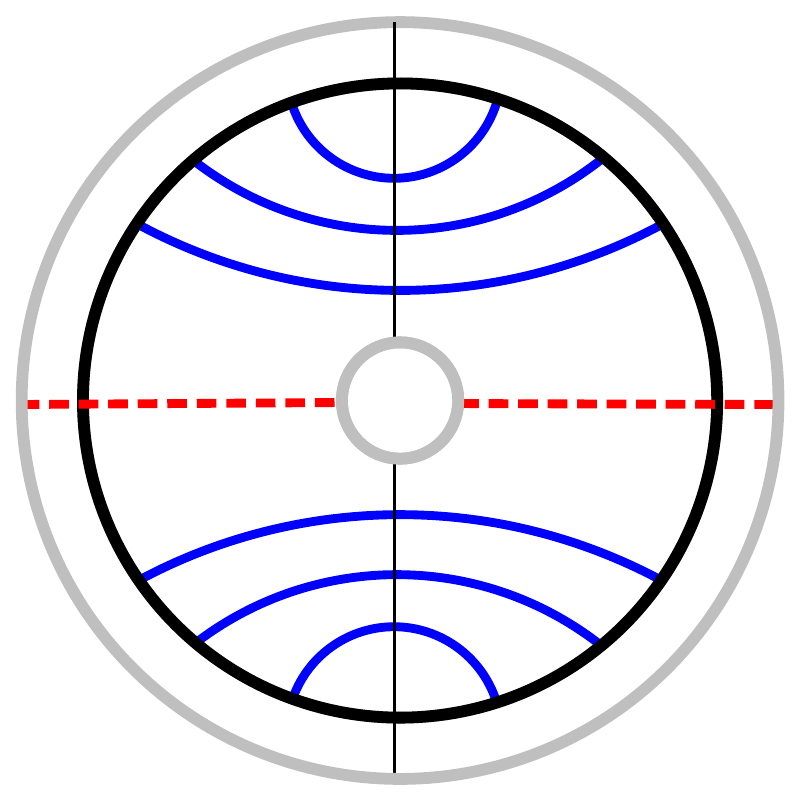}

\put(44,74.5){$a_1$}
\put(33,74.5){$a_1'$}
\put(56,70){$a_2$}
\put(20,70){$a_2'$}
\put(74,43){$a_n$}
\put(3.4,43){$a_n'$}

\put(41,5.15){$a_{2n-1}$}
\put(30,5.15){$a_{2n-1}'$}
\end{overpic} }}
$$
\caption{Bubble reduction.}\label{bubble}
\end{figure}

Recall that $a_n$ is the middle term of the sequence in Definition \ref{basis}. This term will give the factor $S_{a_n}(z)$ in the Gram determinant of generalized type A. Therefore, to see the difference between the Gram determinant of type $A$ and the Gram determinant of generalized type $A$, for fixed $i$, we have to count the number of states with the term which has the value $a_n = n-2i,$ for $0\leq i \leq \frac{n}{2}.$ We find the exponents of $S_{n-2i}(z)$ using the Desir\'{e} - Andr\'{e} reflection principle. We show that this number is equal to $\big( {n \choose i}-{n \choose i-1}\big)^2$. 
This follows from the next lemma. 

\begin{lemma}\label{lattice path}\

Consider all mountain paths from $(0,0)$ to $(n, n - 2i)$ that lie completely in the first quadrant of the $xy$ - plane. The number of such paths is ${n \choose i} - {n \choose i-1}$. 

\end{lemma}

\begin{proof}\

\begin{figure}[ht]
\centering
$$
\vcenter{\hbox{
\begin{overpic}[unit=1mm, scale = .35]{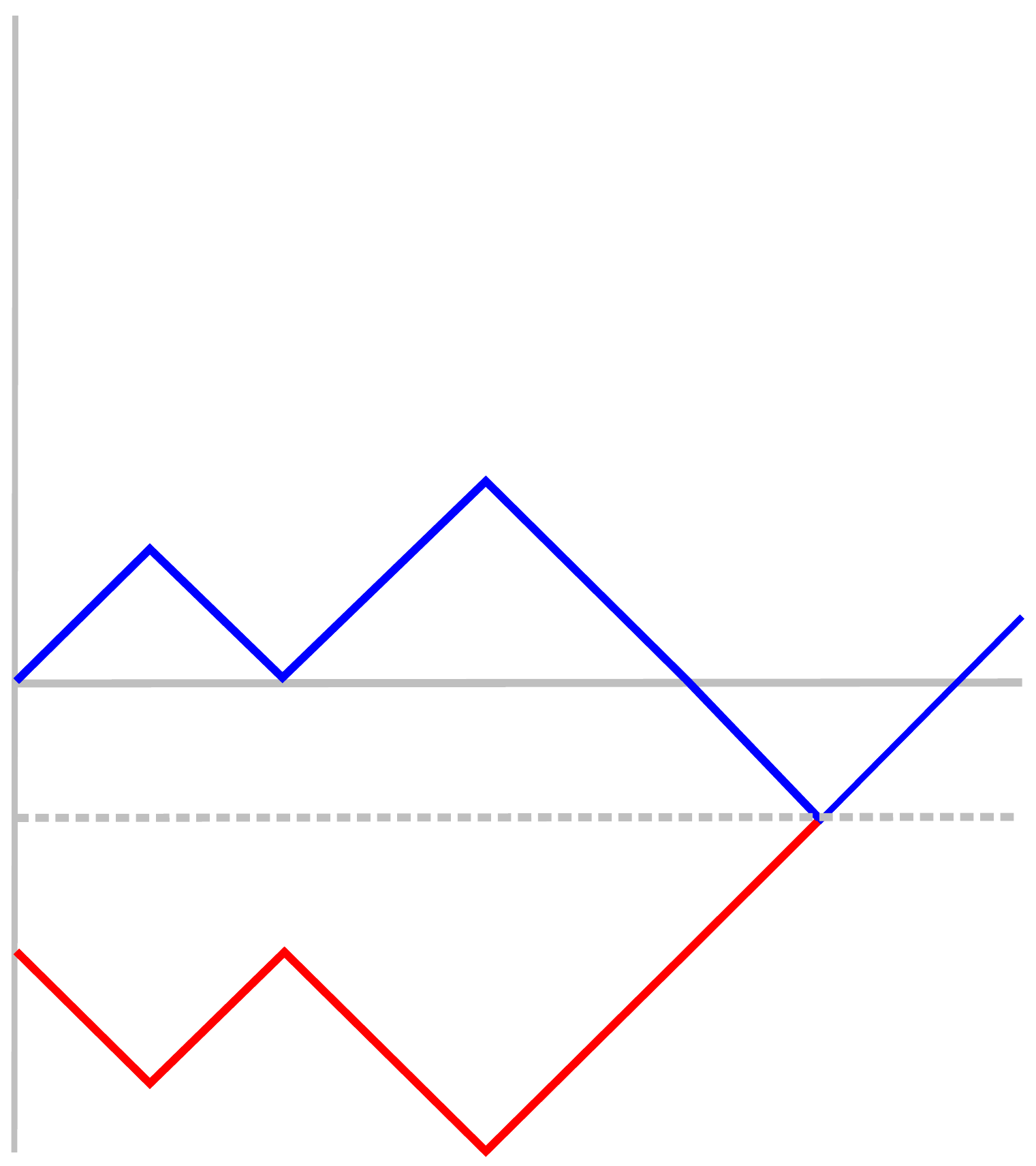}
\put(-6.5, 17){$-1$}
\put(-6.5, 10){$-2$}
\put(-3.65,24){$0$}
\end{overpic} }}
$$
\caption{Mountain paths and their reflection about the horizontal axis $y = -1$.}\label{Mountain}
\end{figure}

Consider all mountain paths from $(0,0)$ to $(n, n - 2i)$. These are paths whose piecewise smooth components have slope $\pm 1$. See Figure \ref{Mountain} for an example. There are ${n \choose i}$ such paths. Note that these paths can have portions that lie below the $x$ - axis. We now find the number of mountain paths that have portions below this axis (`bad' paths). Travel along the path from $(0,0)$ till you reach the point that has the $y$ coordinate equal to $-1$ for the first time. Reflect the path till this point about the horizontal line $y = -1$, keeping the remaining portion of the path untouched. We obtain the path from $(0, -2)$ till $(n,  n - 2i)$.  See Figure \ref{Mountain}. This construction gives a bijection from all such bad paths to all paths from $(0, -2)$ to $(n, n - 2i)$. The number of these paths is $n \choose i-1$. Therefore, the number of paths lying completely above the $x$ - axis is ${n \choose i} - {n \choose i-1}$. 
\end{proof}

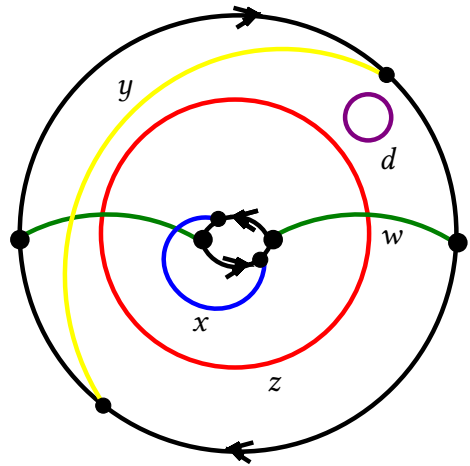
\begin{figure}[ht]
\centering
$
\vcenter{\hbox{
\begin{overpic}[unit=1mm, scale = .5]{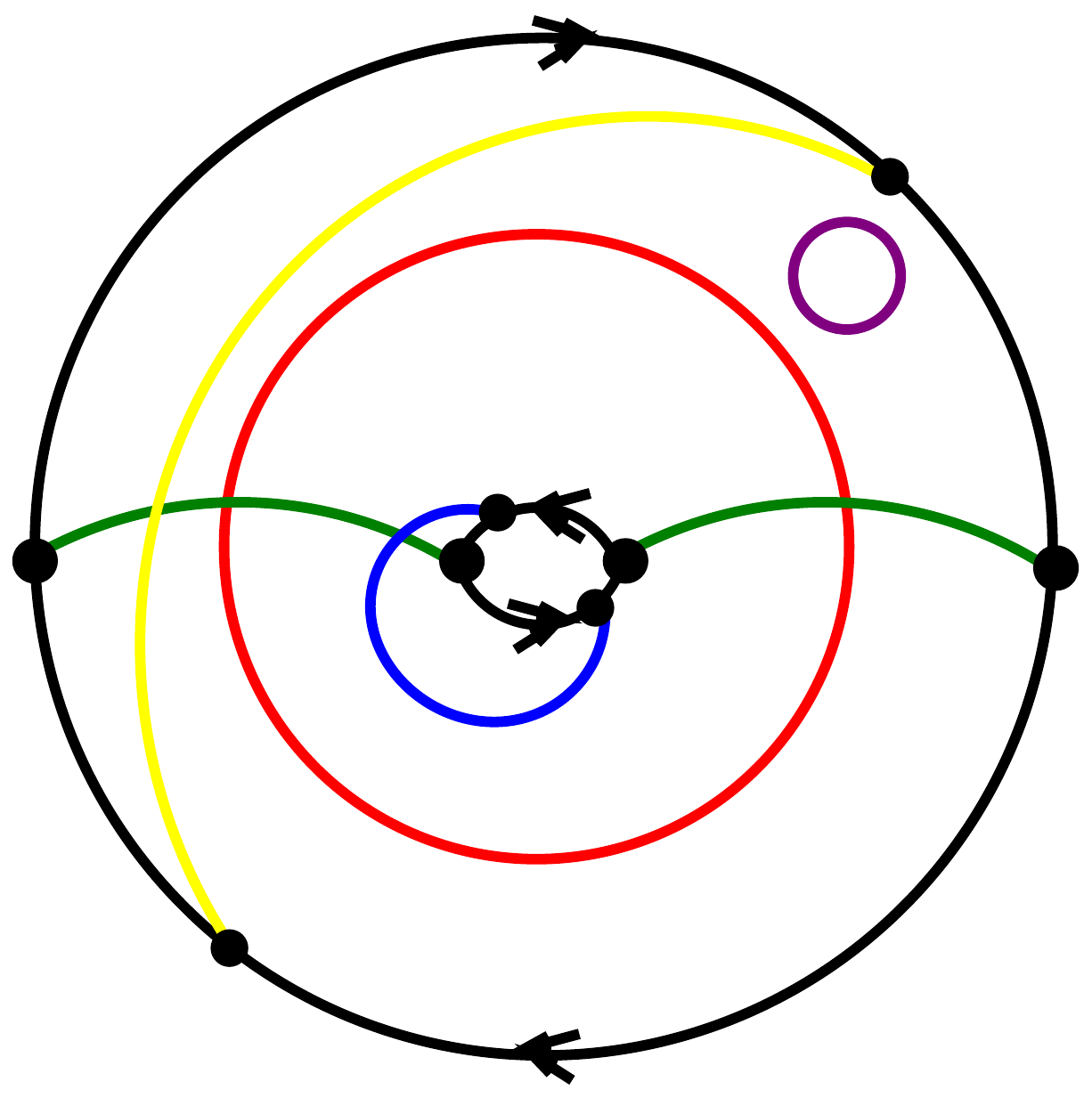}
\put(15,50){$y$}
\put(25, 18.5){$x$}
\put(35,10.5){$z$}
\put(50,30){$w$}
\put(50 ,40){$d$}
\end{overpic} }}
$
\caption{The five different simple closed curves on the Klein bottle. }\label{klein}
\end{figure}

\section{Future Directions: The Gram Determinant of Type Mb}

The fourth author proposed the notion of the Gram determinant of type $Mb$ evaluated in the Klein bottle \cite{Prz3}, that is, the study of the Gram determinant based on crossingless connections on a M\"obius band. In this case, the bilinear form is defined by identifying two M\"obius bands along their boundaries (see Definition \ref{mob}). The identification produces a Klein bottle which has five homotopically distinct curves and thus the Gram matrix given by this bilinear form on the elements of $S_{2, \infty}( Mb \ \hat \times \ I,\{x_i\}_1^{2n})$ has entries in the five variables that denote these curves.  After some preliminary calculations, Chen proposed the following
conjecture \cite{Che} for the Gram determinant of the M\"obius band.

\begin{conjecture}[Chen] \label{Qi}\

Let $R = \mathbb{Z}[A^{\pm 1}].$ Then,
the Gram determinant of type $Mb$ for $n \geq 1$, denoted $D_n^{(Mb)}$, is

 \begin{eqnarray*}
D^{(Mb)}_n(d,w,x,y,z) &=& \prod_{k=1}^n (T_k(d)+(-1)^kz)^{\binom{2n}{n-k} } \\
& & \prod_{k=1 \atop k \text{ odd }}^n (T_k(d) - (-1)^k z)T_k(w) -2xy)^{\binom{2n}{n-k}} \\
& & \prod_{ k=1 \atop k \text{ even }}^n (T_k(d) - (-1)^kz)T_k(w)-2(2-z))^{\binom{2n}{n-k}} \\
& & \prod_{i=1}^n D_{n,i},
\end{eqnarray*}
where $D_{n,i} = \prod\limits_{k=1+i}^n (T_{2k}(d)-2)^{\binom{2n}{n-k}}$, and $i$ represents the number of curves passing through the crosscap.
\end{conjecture}

In the remaining part of this paper, we discuss some of our work on this project and state a few results supporting Conjecture \ref{Qi} (Proposition \ref{p}, Theorem \ref{P1}, and Theorem \ref{P2}). Compare Theorem \ref{P1} and Theorem \ref{P2} with analogous cases in \cite{Ch-P} and \cite{P-Zh}, respectively.  

\begin{definition}\label{mob}\

Consider a M\"obius band with $2n$ marked points on its boundary, and $n$ crossingless curves connecting these $2n$ marked points. Denote the set of all diagrams of crossingless  connections on the M\"obius band with $2n$ points by $Mb_n$. 
Then the number of distinct crossingless connections is, 
$$ |Mb_n| = \sum_{k=0}^{n} \binom{2n}{k} = 2^{2n-1} + \frac{1}{2} {2n \choose n}. $$
Notice that in the formula $\sum\limits_{k=0}^{n} {2n \choose k}$, $k$ tells us how many curves are disjoint from the crosscap (see Figure \ref{ex1}). 

Define a bilinear form $\langle \ , \ \rangle$ in the following way:
$$ \langle \ , \ \rangle : \mathcal{S}_{2,\infty}(Mb\  \hat \times \ I, \{x_i\}_1^{2n}) \times \mathcal{S}_{2,\infty}(Mb\ \hat \times \  I, \{x_i\}_1^{2n}) \longrightarrow \mathbb{Z}[d, w,x,y,z],$$
where $d, w, x, y,$ and $z$ denote the five homotopically distinct simple closed curves on the Klein bottle which are illustrated in Figure \ref{klein}.\footnote{The variable $d$ denotes the homotopically trivial curve, while $x$ intersects the inner crosscap, $y$ intersects the outer crosscap, $z$ separates the two crosscaps, and $w$ intersects both crosscaps.}

Given $m_i, m_j \in Mb_n$, identify the boundary component of $m_i$ with that of the inversion of $m_j$, respecting the labels of the marked points. The result is a collection of disjoint simple closed curves on the Klein bottle which are of five different types, as shown in Figure \ref{klein}. In particular, given two elements in $Mb_n$, the image of its bilinear form is a monomial in $\mathbb{Z}[d,x,y,z,w]$, that is, $\langle m_i , m_j\rangle =  d^mx^ny^kz^lw^h$ where $m,n,k,l$ and $h$ denote the number of these simple closed curves, respectively. 

The Gram matrix of type $Mb$ is defined as $\ G_n^{Mb} = (\langle m_i , m_j\rangle)_{1 \leq i, j \leq |Mb_n|}$ and its determinant is denoted by $D_n^{Mb}$. Since the Gram matrix is given by the bilinear form, the size of the matrix, $|G_n^{Mb}|$, is equal to the number of different crossingless connections.

\end{definition}

\begin{example}\

For $n=2$, there are eleven different crossingless connections in the M\"obius band. 
\begin{figure}[ht]
\centering
$$ \left\{ \vcenter{\hbox{\includegraphics[scale = .15]{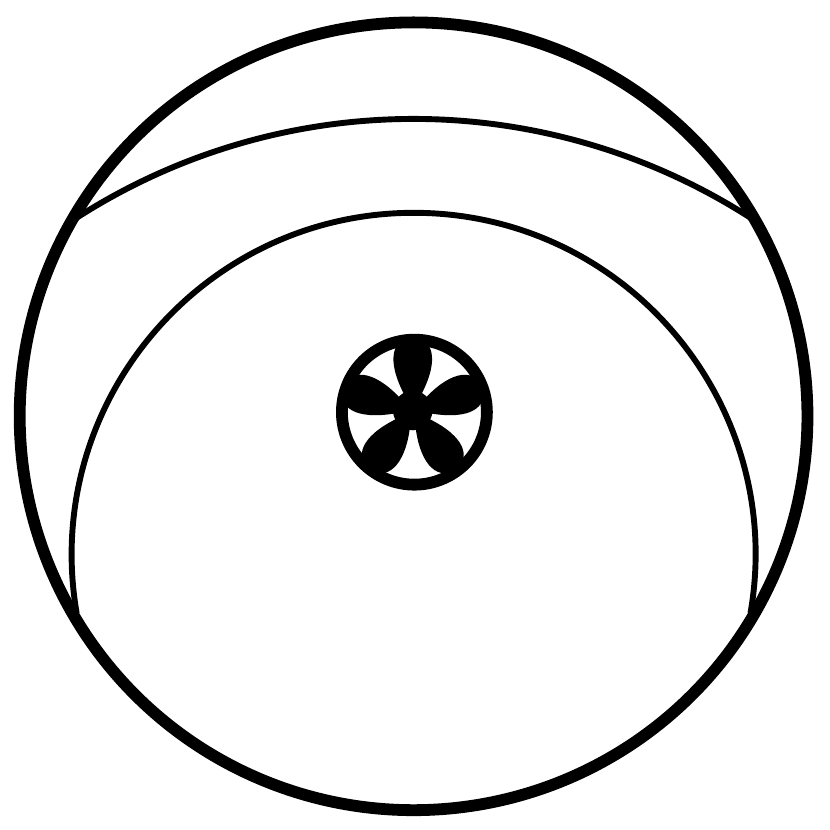}}} ,
\vcenter{\hbox{\includegraphics[scale = .15]{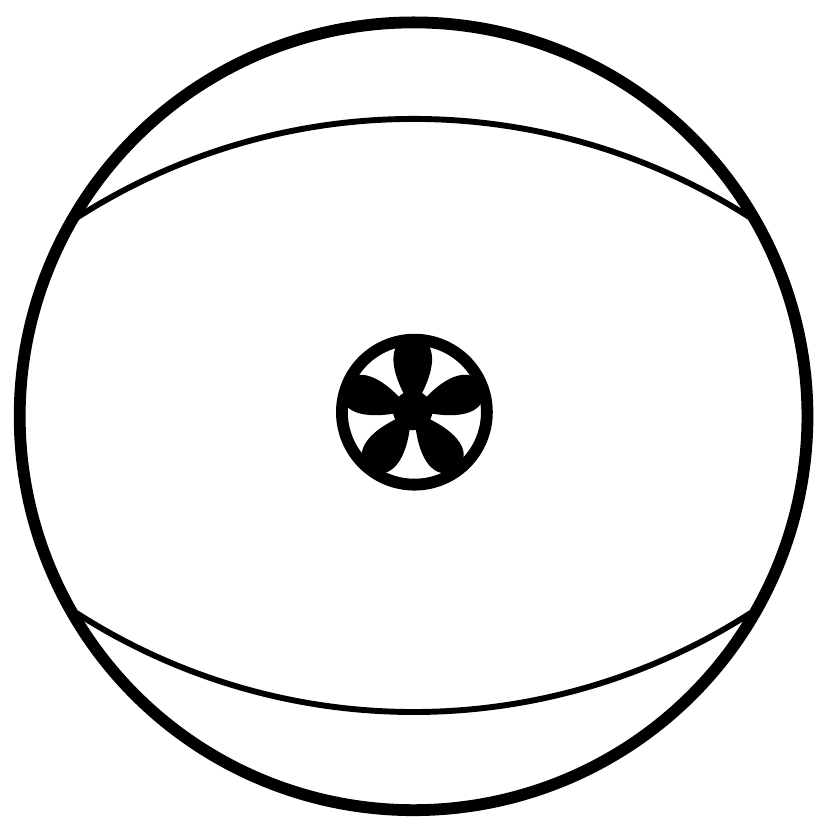}}},  
\vcenter{\hbox{\includegraphics[scale = .15]{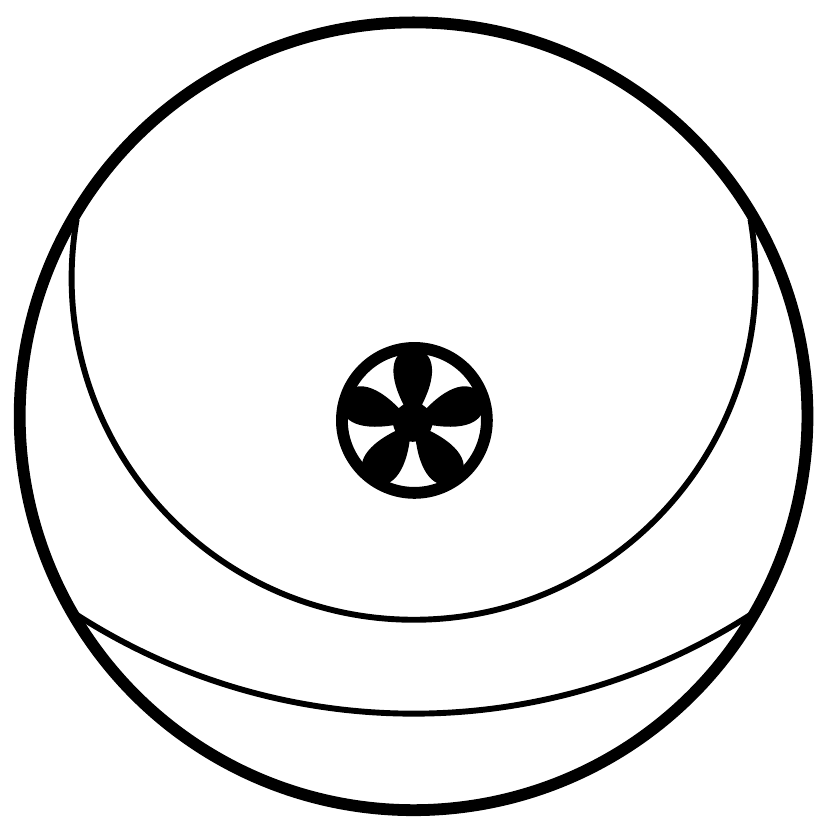}}},
\vcenter{\hbox{\includegraphics[scale = .15]{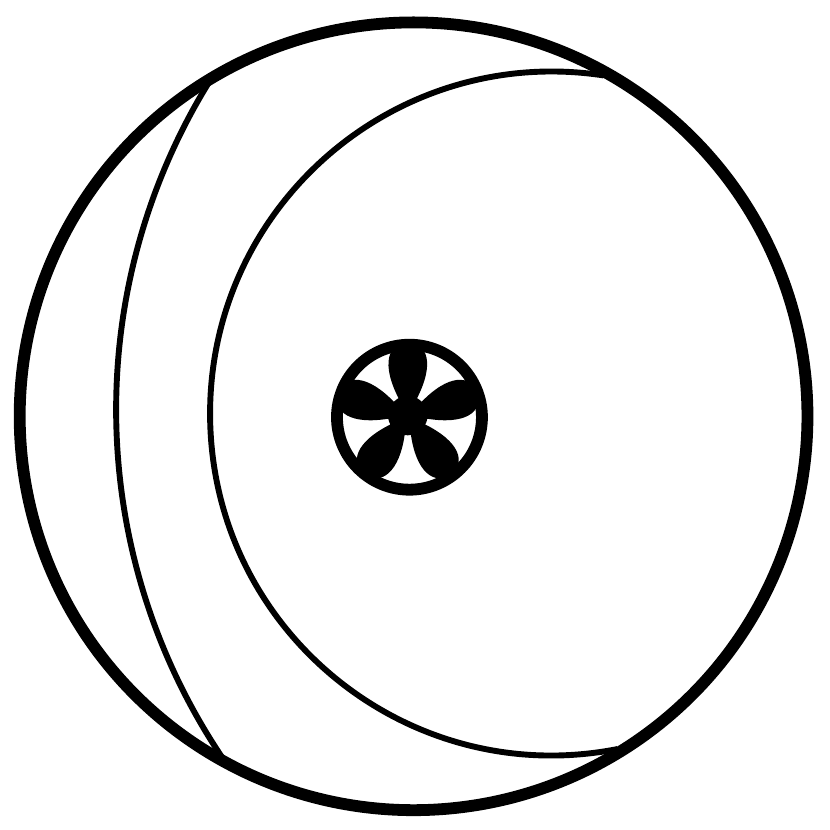}}}, 
\vcenter{\hbox{\includegraphics[scale = .15]{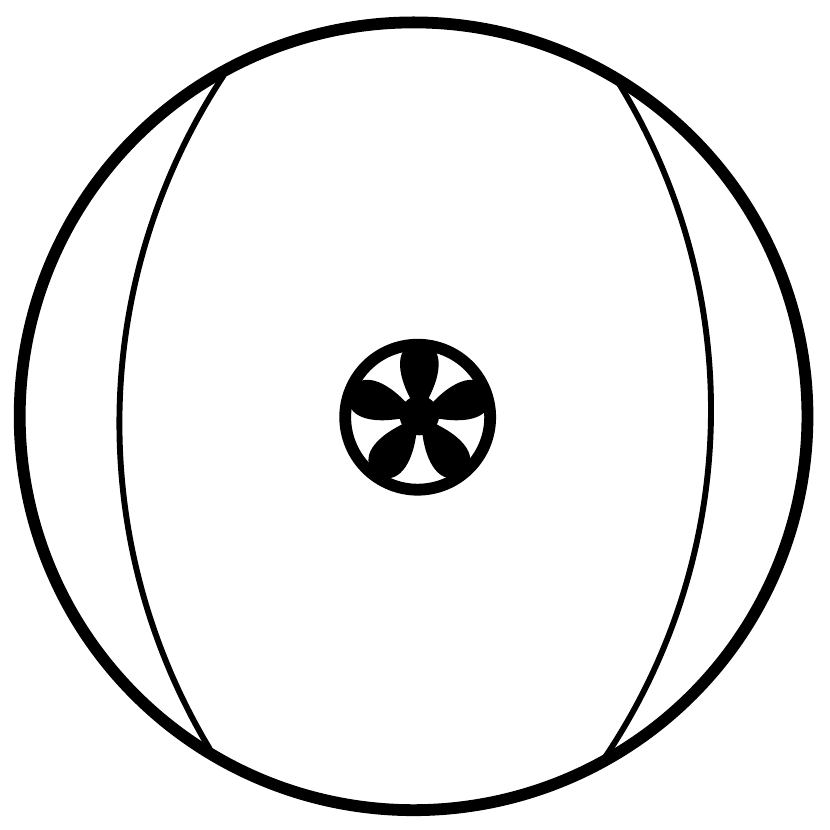}}} , 
\vcenter{\hbox{\includegraphics[scale = .15]{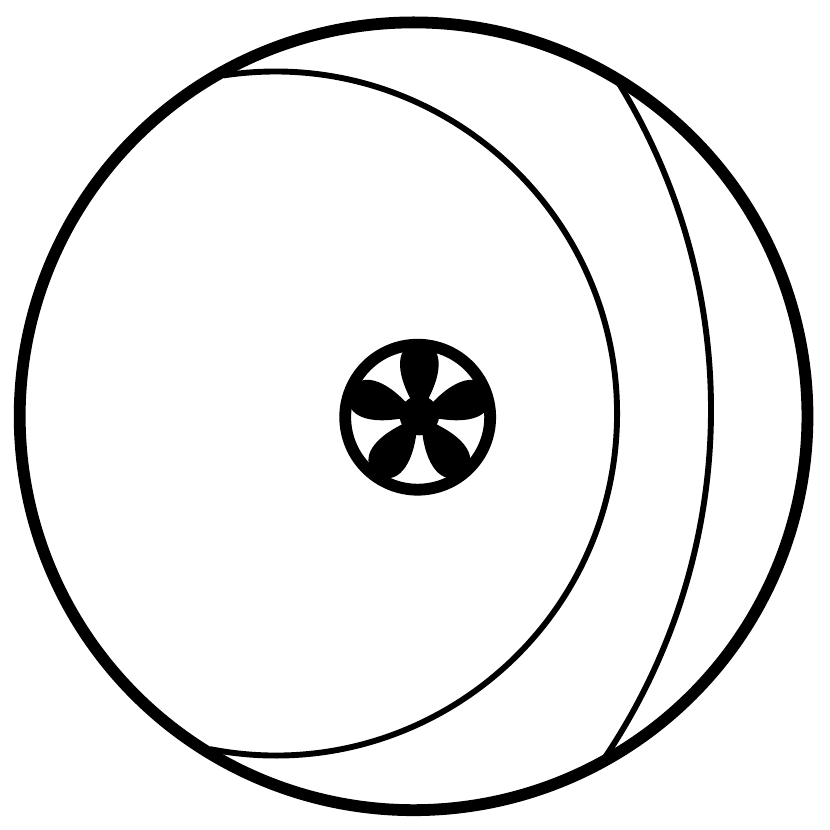}}},
\vcenter{\hbox{\includegraphics[scale = .15]{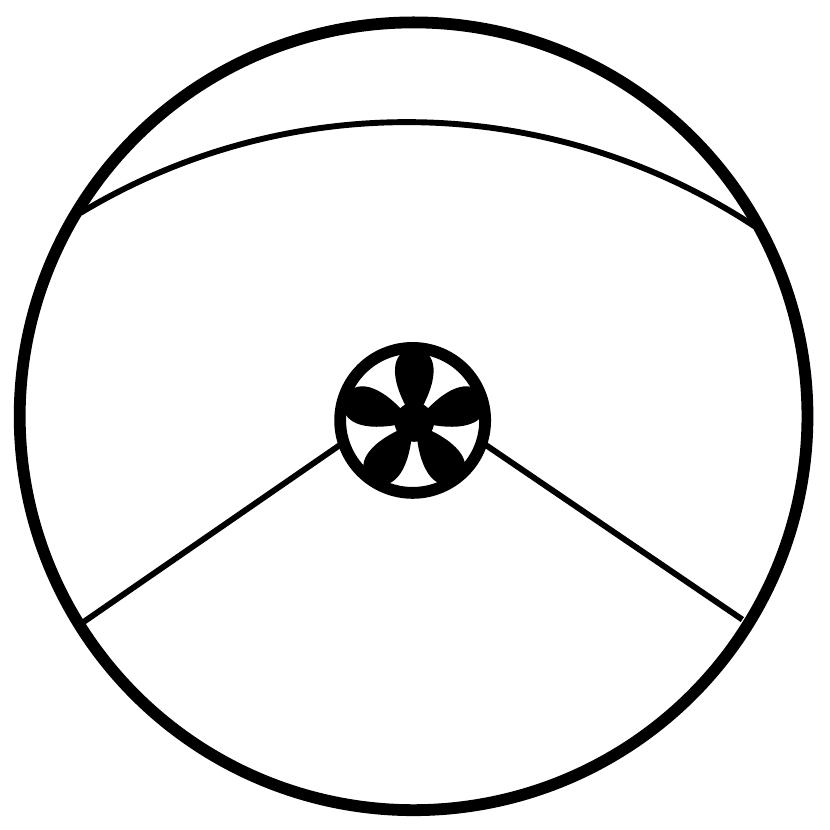}}}, 
\vcenter{\hbox{\includegraphics[scale = .15]{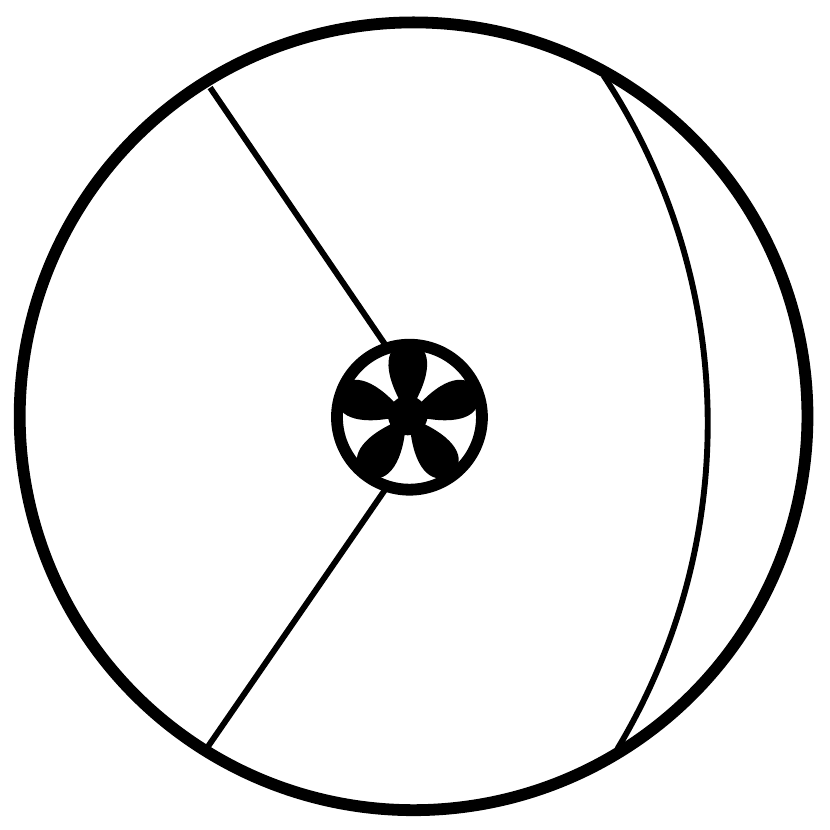}}} , 
\vcenter{\hbox{\includegraphics[scale = .15]{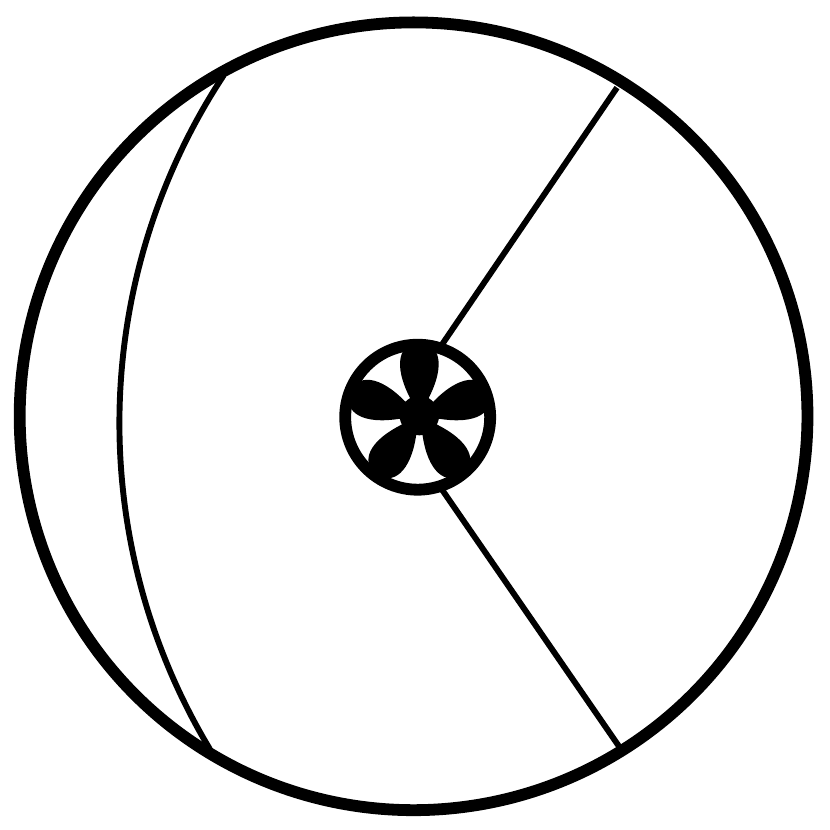}}}, 
\vcenter{\hbox{\includegraphics[scale = .15]{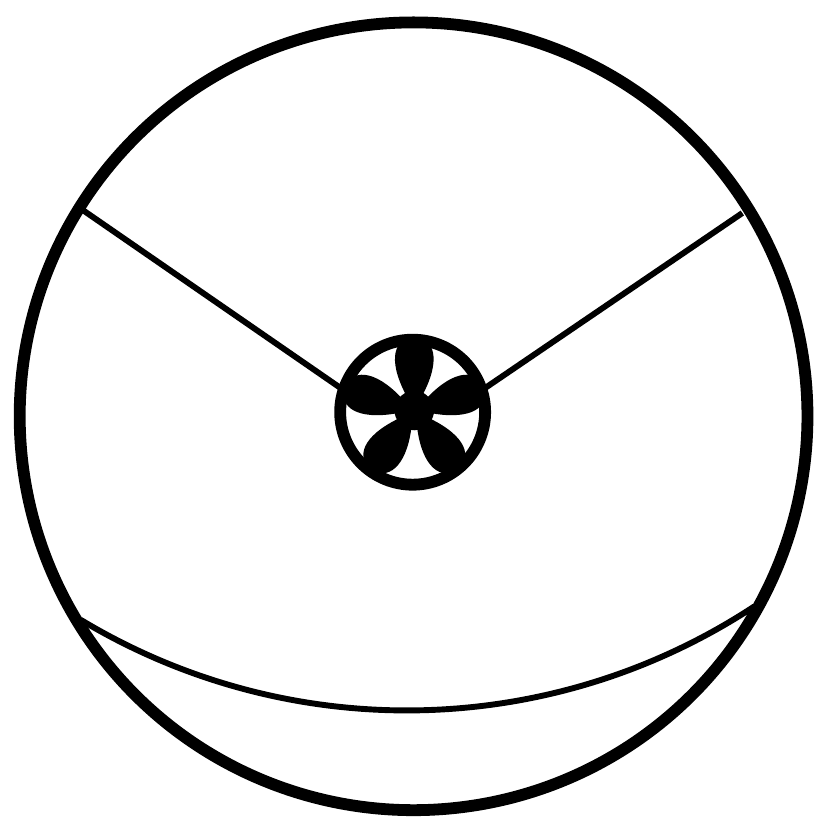}}} , 
\vcenter{\hbox{\includegraphics[scale = .15]{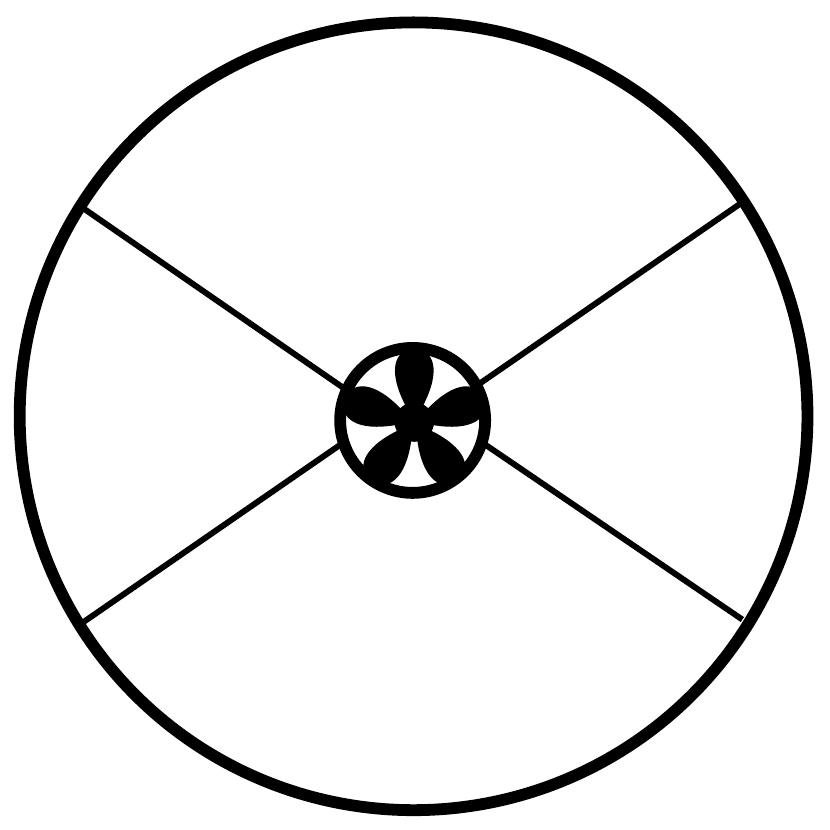}}} \right\} $$
\caption{
In the eleven different crossingless connections M\"obius band, the first six, $\binom{4}{2}$, do not intersect the crosscap, the next four, $ \binom{4}{1}$, pass through the crosscap once, and finally, the last one, $ \binom{4}{0}$, passes through the crosscap twice.}
\label{ex1}
\end{figure}

The Gram matrix given by the bilinear form is as follows: \\ \ 

{\centering
\resizebox{\columnwidth}{!}{%
\begin{tabular}{ c|||c|c|c|c|c|c||c|c|c|c||c}
 $\langle \ , \ \rangle$ & $\vcenter{\hbox{\includegraphics[scale = .13, height = 0.6cm]{GMB2_2.pdf}}}$ 
& $\vcenter{\hbox{\includegraphics[scale = .13,height = 0.6cm]{GMB2_1.pdf}}}$
& $\vcenter{\hbox{\includegraphics[scale = .13,height = 0.6cm]{GMB2_3.pdf}}}$
& $\vcenter{\hbox{\includegraphics[scale = .13,height = 0.6cm]{GMB2_4.pdf}}}$
& $\vcenter{\hbox{\includegraphics[scale = .13,height = 0.6cm]{GMB2_6.pdf}}}$
& $\vcenter{\hbox{\includegraphics[scale = .13,height = 0.6cm]{GMB2_5.pdf}}}$
& $\vcenter{\hbox{\includegraphics[scale = .13,height = 0.6cm]{GMB2_10.pdf}}}$
& $\vcenter{\hbox{\includegraphics[scale = .13,height = 0.6cm]{GMB2_8.pdf}}}$
& $\vcenter{\hbox{\includegraphics[scale = .13,height = 0.6cm]{GMB2_7.pdf}}}$
& $\vcenter{\hbox{\includegraphics[scale = .13,height = 0.6cm]{GMB2_9.pdf}}}$
& $\vcenter{\hbox{\includegraphics[scale = .13,height = 0.6cm]{GMB2_11.pdf}}}$ \\
\hline \hline \hline 	
$\vcenter{\hbox{\includegraphics[scale = .13, height = 0.6cm]{GMB2_2.pdf}}}$        
& $ d^2$ & $ dz$ & $z^2$ & $z$ & $d$ & $z$ & $dy$ & $y$ & $yz$ & $ y$ & $z $ \\ \hline 
$\vcenter{\hbox{\includegraphics[scale = .13, height = 0.6cm]{GMB2_1.pdf}}}$          
& $ dz$ & $ d^2$ & $dz$ & $d$ & $z$ & $d$ & $dy$ & $y$ & $dy$ & $ y$ & $d $ \\ \hline 
$\vcenter{\hbox{\includegraphics[scale = .13, height = 0.6cm]{GMB2_3.pdf}}}$ 
& $ z^2$ & $ dz$ & $d^2$ & $z$ & $d$ & $z$ & $yz$ & $y$ & $dy$ & $ y$ & $z $ \\ \hline 
$\vcenter{\hbox{\includegraphics[scale = .13, height = 0.6cm]{GMB2_4.pdf}}}$  
& $ z$ & $ d$ & $z$ & $d^2$ & $dz$ & $z^2$ & $y$ & $yz$ & $y$ & $ dy$ & $z $ \\ \hline 
$\vcenter{\hbox{\includegraphics[scale = .13, height = 0.6cm]{GMB2_6.pdf}}}$  
& $ d$ & $ z$ & $d$ & $dz$ & $d^2$ & $dz$ & $y$ & $dy$ & $y$ & $ dy$ & $d $ \\ \hline 
$\vcenter{\hbox{\includegraphics[scale = .13, height = 0.6cm]{GMB2_5.pdf}}}$  
& $ z$ & $ d$ & $z$ & $z^2$ & $dz$ & $d^2$ & $y$ & $dy$ & $y$ & $ yz$ & $z $ \\ \hline \hline
$\vcenter{\hbox{\includegraphics[scale = .13, height = 0.6cm]{GMB2_10.pdf}}}$  
& $ dx$ & $ dx$ & $xz$ & $x$ & $x$ & $x$ & $dw$ & $w$ & $xy$ & $ w$ & $x $ \\ \hline 
$\vcenter{\hbox{\includegraphics[scale = .13, height = 0.6cm]{GMB2_8.pdf}}}$  
& $x$ & $ x$ & $x$ & $xz$ & $dx$ & $dx$ & $w$ & $dw$ & $w$ & $ xy$ & $x $ \\ \hline 
$\vcenter{\hbox{\includegraphics[scale = .13, height = 0.6cm]{GMB2_7.pdf}}}$  
& $ xz$ & $ dx$ & $dx$ & $x$ & $x$ & $x$ & $xy$ & $w$ & $dw$ & $ w$ & $x $ \\ \hline 
$\vcenter{\hbox{\includegraphics[scale = .13, height = 0.6cm]{GMB2_9.pdf}}}$  
& $ x$ & $ x$ & $x$ & $dx$ & $dx$ & $xz$ & $w$ & $xy$ & $w$ & $ dw$ & $x $ \\ \hline \hline
$\vcenter{\hbox{\includegraphics[scale = .13, height = 0.6cm]{GMB2_11.pdf}}}$  
& $ z$ & $ d$ & $z$ & $z$ & $d$ & $z$ & $y$ & $y$ & $y$ & $ y$ & $w^2 $ 
\end{tabular}}}
 \\ \ \\ \ \\ \ 
The determinant of the Gram matrix $Det(G_{2}^{Mb} )$ is equal to :
 $$(d-z)^4[(d+z)w-2xy]^4(d^2(d^2-4))(d^2-2+z)[(d^2-2-z)(w^2-2)-2(2-z)].$$
 In terms of Chebyshev polynomials of the first kind, the formula coincides with Chen's conjecture and  $Det(G_{2}^{Mb} )$ is equal to: 
$$ (T_1(d)-z)^4[(T_1(d)+z)T_1(w)-2xy]^4(T_4(d)-2)(T_2(d)+z)[(T_2(d)-z)T_2(w)-2(2-z)] $$ 

\end{example}

\begin{figure}[ht]
\centering
$$
\vcenter{\hbox{
\begin{overpic}[unit=1mm, scale = .4]{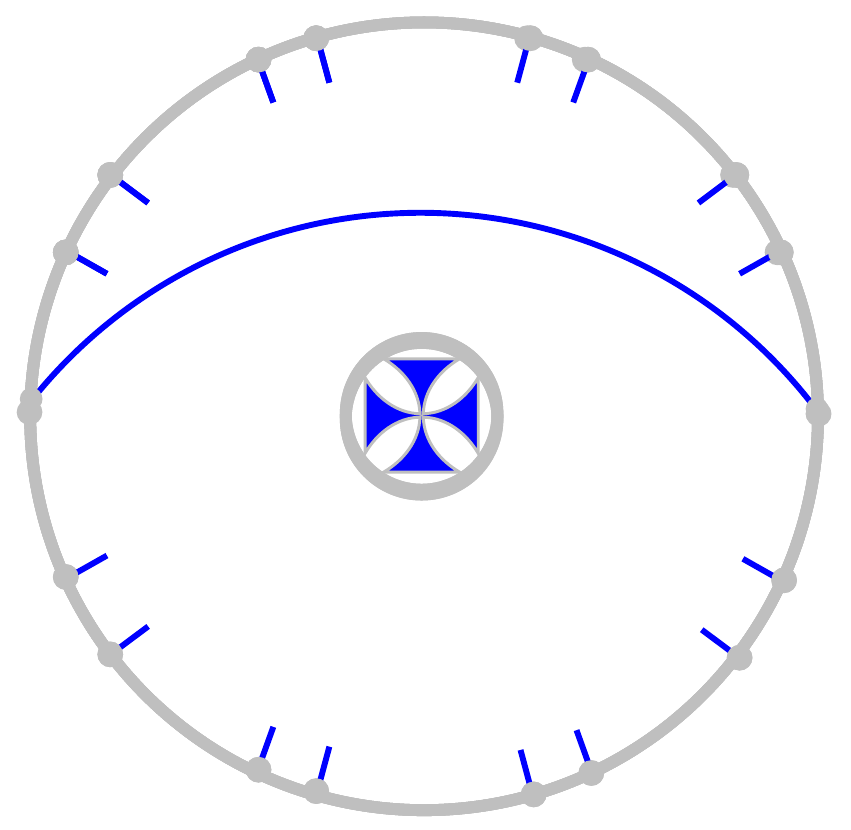}
\put(16,-3.5){$b_1$}

\end{overpic} }}  \ \ \ \ \ \ 
\vcenter{\hbox{
\begin{overpic}[unit=1mm, scale = .4]{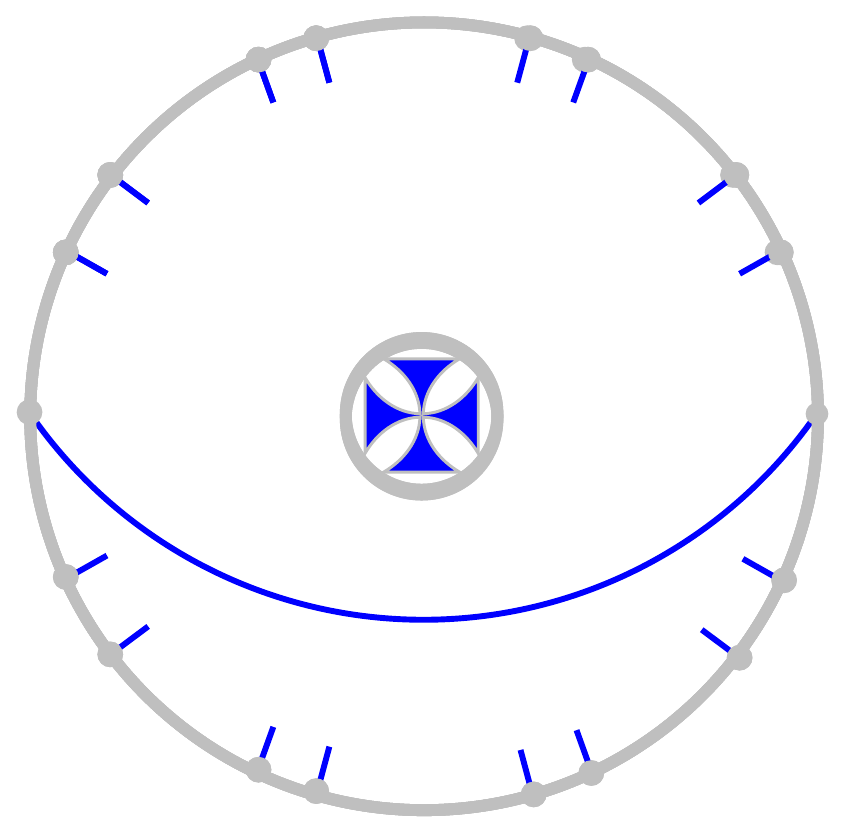}

\put(16,-3.5){$b_2$}

\end{overpic} }}  \ \ \ \ \ \ 
\vcenter{\hbox{
\begin{overpic}[unit=1mm, scale = .4]{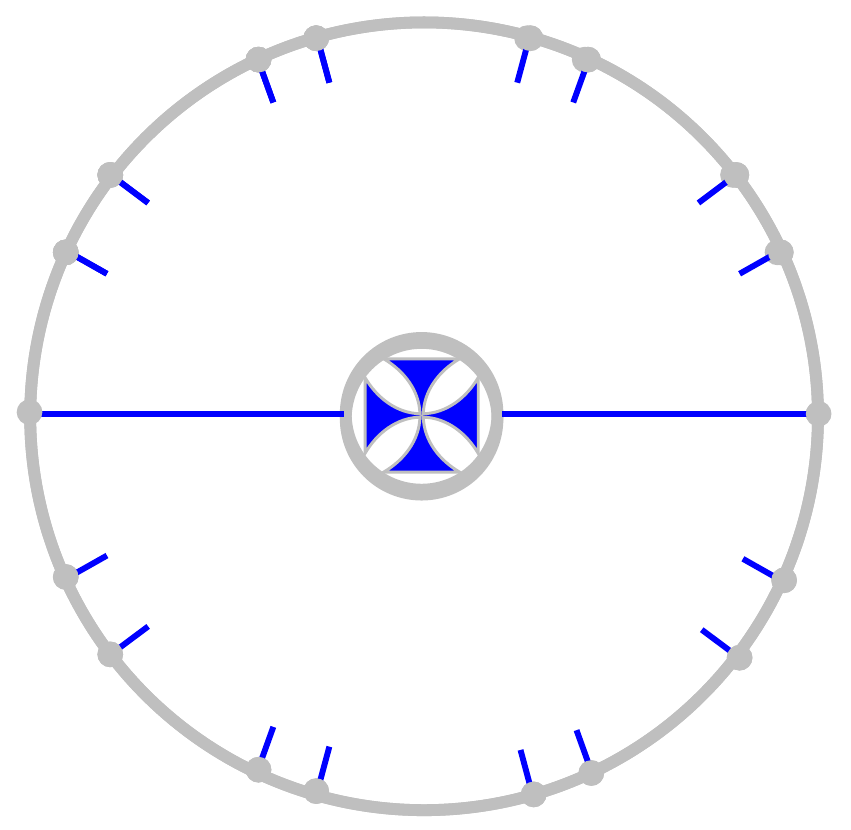}

\put(16,-3.5){$b_3$}
\put(24,19){$\alpha_3$}

\end{overpic} }}
$$
\caption{The crossingless connection $b_3 \in Mb_n$ with the connection $\alpha_3$ in $b_3$ passing through the crosscap and the crossingless connections $b_1,b_2 \in Mb_n$ obtained from $b_3$ by modifying $\alpha_3$ so that it does not pass through the crosscap.} \label{newfigure}
\end{figure}

\begin{proposition}\label{p}\ 

$Det(G^{Kb}_n)$  is divisible by $(w(d+z)-2xy)^{2n \choose n-1}$.

\end{proposition}

\begin{proof}\

Consider a crossingless connection, that is, an element $b_3\in Mb_n$, which cuts the crosscap exactly once (denoted by $\alpha_3$). Let $b_1$ and $b_2$ be obtained from $b_3$ by modifying $\alpha_3$ so that it does not pass through the crosscap (see Figure \ref{newfigure}). Consider the block associated with these three elements.  For some monomial $u,$ we get:
\[
        \left[
                \begin{array}{ccc}
                        ud  &  uz &  uy \\
                        uz  &  ud &  uy  \\
                        ux  &  ux & u w
\end{array} \right]
        \]
If we replace $b_3$ by $(d+z)b_3-y(b_1+b_2)$, denoted by $b_3',$ we obtain the matrix:
\[
        \left[
                \begin{array}{ccc}
                        ud  &  uz &  0 \\
                        uz  &  ud &  0  \\
                        ux  &  ux & u((d+z)w-2xy)
\end{array} \right]
        \]
with determinant $(d+z)$ times the determinant of the previous matrix. Now, we observe that for any element $c \in Mb_n$ the row $\langle c,b_1\rangle,\langle c,b_2 \rangle, \langle c,b_3 \rangle$ is proportional (by a monomial) to $(d,z,y)$, $(z,d,y)$, or $(x,x,y).$ Thus, the column operation we did earlier produces either a zero or a monomial multiple of $((d+z)w-2xy)$. Therefore, the column under $b_3'$ is divisible by
$((d+z)w-2xy)$. Now, we can perform the same operation, as on $b_3,$ on any element (crossingless connection) which cuts the crosscap exactly once. There are ${2n \choose n-1}$ such elements and thus we obtain a matrix with ${2n \choose n-1}$ columns divisible by $((d+z)w-2xy)$. Additionally, note that $d+z$ and $((d+z)w-2xy)$ are co-prime. Therefore, $Det(G^{Kb}_n)$  is divisible by $(w(d+z)-2xy)^{2n \choose n-1}$.
\end{proof}

\begin{theorem} \label{P1}
$$\prod\limits_{k=1}^n(T_k(d)+(-1)^k z)^{2n \choose n-k} \mbox{ divides } D^{(Mb)}_n.$$
\end{theorem}

\begin{theorem} \label{P2}
$$(D^{(Mb)}_{n-1})^2 \text{  divides } D^{(Mb)}_n \text{,  if } (d^2-1) \text{ is invertible.} $$ 
\end{theorem}

\section*{Acknowledgments}

The third author was partially supported by the CCAS Dean's Dissertation Completion Fellowship. The fourth author was partially supported by the Simons Foundation Collaboration Grant for Mathematicians - 316446 and the CCAS Dean’s Research Chair award.

\end{document}